\documentclass{amsart}
\usepackage{amssymb}

\newcommand{\Id}{\operatorname{Id}}
\newcommand{\Ass}{\operatorname{Ass}}
\newcommand{\supp}{\operatorname{supp}}
\newcommand{\Spec}{\operatorname{Spec}}
\newcommand{\Coker}{\operatorname{Coker}}
\newcommand{\Tor}{\operatorname{Tor}}
\newcommand{\Rad}{\operatorname{Rad}}
\newcommand{\Soc}{\operatorname{Soc}}
\newcommand{\Top}{\operatorname{Top}}
\newcommand{\Ker}{\operatorname{Ker}}
\newcommand{\Hom}{\operatorname{Hom}}
\newcommand{\Ext}{\operatorname{Ext}}
\newcommand{\Mod}{\operatorname{Mod}}
\newcommand{\fgMod}{\operatorname{mod}}
\newcommand{\End}{\operatorname{End}}
\newcommand{\tr}{\operatorname{tr}}
\newcommand{\id}{\operatorname{id}}
\newcommand{\mattr}{\operatorname{Tr}}
\renewcommand{\dim}{\operatorname{dim}}
\newcommand{\Irr}{\operatorname{Irr}}
\newcommand{\wt}{\operatorname{wt}}
\newcommand{\Eu}{\operatorname{Eu}}
\newcommand{\gr}{\operatorname{gr}}

\newcommand{\C}{\mathbb C}
\newcommand{\Z}{\mathbb Z}
\newcommand{\A}{\mathcal A}
\newcommand{\bk}{\mathbf k}
\newcommand{\bh}{\mathbf h}

\renewcommand{\H}{\mathcal H}
\renewcommand{\L}{\mathcal L}
\renewcommand{\Im}{\operatorname{Im}}

\newtheorem{theorem}{Theorem}[section]
\newtheorem{corollary}[theorem]{Corollary}
\newtheorem{lemma}[theorem]{Lemma}
\newtheorem{proposition}[theorem]{Proposition}

\theoremstyle{definition}
\newtheorem{definition}[theorem]{Definition}
\newtheorem{remark}[theorem]{Remark}
\newtheorem{example}[theorem]{Example}

\begin{document}
\title{Finite dimensional Hecke algebras} 
\author{Susumu Ariki}

\begin{abstract}
This article surveys development on finite dimensional Hecke algebras
in the last decade.
In the first part, we explain results on canonical basic sets by 
Geck and Jacon and propose a categorification framework which
is suitable for our example of Hecke algebras.
In the second part, we review basics of Kashiwara crystal 
and explain the Fock space theory of cyclotomic Hecke algebras and 
its applications. In the third part, we explain Rouquier's theory 
of quasihereditary covers of cyclotomic Hecke algebras. We add detailed 
explanation of the proofs here. 
The third part is based on my intensive course given at Nagoya university 
in January 2007. 
\end{abstract}

\address{S.A.: Research Institute for Mathematical Sciences, 
Kyoto University, Kyoto 606-8502, Japan}
\email{ariki@kurims.kyoto-u.ac.jp}

\subjclass[2000]{Primary 17B37; Secondary 20C08,05E99}

\maketitle

\section{Introduction}

In this article, we will explain current views on the modular
representation theory of finite dimensional Hecke algebras.
In the last decade, I followed the approach by Dipper and James,
and, it has become clear that the language from solvable lattice models,
which uses terminology like Fock spaces, crystal bases, etc,
is one of the most natural. On the other hand, Geck followed Lusztig's
approach and applied his methods to the modular representation theory
of Hecke algebras. When we consider Hecke algebras of type $A$ and $B$,
or more generally, cyclotomic Hecke algebras of classical type,
the two approaches interact, and the study of various labelling sets of
irreducible modules has stimulated an interest on cellular
algebra structures on the Hecke algebras.
In type $A$, we have theory of Specht and dual Specht modules.
In type $B$, based on their earlier work \cite{BI} and \cite{GI},
Bonnaf\'e, Geck, Iancu and Lam \cite{BGIL} have conjectured in a very
precise manner how Kazhdan-Lusztig cells should give various cellular
algebra structures, and Geck, Iancu and Pallikaros \cite{GIP} showed
that the known cellular structure given by Dipper, James and Murphy is one of them.
This suggests us a categorification framework for integrable highest weight
$U_v(\hat{sl}_e)$-modules with two specializations at $v=0$ and
$v=1$.

In type $A$, we have $q$-Schur algebras, which has been an object of
intensive study in the last several decades. By Leclerc-Thibon \cite{LT} and
Varagnolo-Vasserot \cite{VV}, the algebras also fit well in the
categorification picture. Note that $q$-Schur algebras are
cellular algebras of quasihereditary type. When the base field
is $\C$, Rouquier has showed that the category $\mathcal O$ for
the rational Cherednik algebra associated with the symmetric group
is the $q$-Schur algebra \cite{R}, and quasihereditary structures of
$\mathcal O$ explains in some sense why we have the Specht and the dual Specht module
theory. The result depends on his earlier work, one with
Ginzburg, Guay and Opdam \cite{GGOR}, the other with Brou\'e and Malle \cite{BMR}.
Observe that every piece that appears in the above story
has its cyclotomic analogue. Hence, it is natural to
expect that cyclotomic analogue of the story would be true.
This is our current motivation of research, and
even in type $B$, we have not reached a complete understanding, yet.

The paper is organized as follows.
In the first part, we introduce the Hecke algebra and briefly explain
results on the canonical basic sets by Geck and Jacon and
the categorification framework in which the results sit in.
To know more about the canonical basic sets and Hecke algebras,
I recommend reading the survey \cite{G1}.
In the second part, we explain the Fock space
theory, mostly developed by the author and my collaborators,
after explaining Kashiwara crystal. Its applications to Hecke algebras
include the modular branching rule, the representation type, etc.
In the third part, we explain Rouquier's theory of quasihereditary
covers in terms of the category $\mathcal O$ for the rational 
Cherednik algebra. I reorganized the contents of \cite{GGOR} and 
\cite{R} and explain the shortest way to reach the Rouquier's
result. The reader who have read \cite{GGOR} and \cite{R} seriously 
would find that the proofs explained here are very reader friendly.
I hope that this part provides
a good preparation for reading \cite{GGOR} and \cite{R}.

The third part is based on my intensive course given at Nagoya
university in January 2007. At the time, Shoji requested some
written material of the lectures, and Kuwabara asked me if
it could be in English. The third part partially answers
their requests. I thank Shoji for inviting me to give
the course. 

\section{Hecke algebras with unequal parameters}
\subsection{The algebra}

Following \cite{L}, we introduce the Hecke algebra, our main
object of study.

\begin{definition}
We say that $(W,S,L)$ is a \emph{weighted Coxeter group} if 
\begin{itemize}
\item[(i)]
$(W,S)$ is a Coxeter group, $w\mapsto\ell(w)$, the length function.
\item[(ii)]
$L:W\rightarrow\Z$ is such that $L(ww')=L(w)+L(w')$ if 
$\ell(ww')=\ell(w)+\ell(w')$. 
\end{itemize}
\end{definition}

\begin{remark}
Recall that
$$
(st)^{m_{st}/2}=(ts)^{m_{st}/2}\;\text{if $m_{st}$ is even,}\;\; 
(st)^{(m_{st}-1)/2}s=(ts)^{(m_{st}-1)/2}t\;\text{if $m_{st}$ is odd,} 
$$
is part of the defining relations of $W$. 
Giving $L$ is the same as giving a set of values 
$\{L(s)\mid s\in S\}$ with the property that $L(s)=L(t)$ whenever $m_{st}$ is odd. 
\end{remark}

We say that $(W,S)$ is of finite type if $W$ is a finite group. $(W,S)$ 
is isomorphic to product of irreducible Coxeter groups of finite type, 
and the irreducible Coxeter groups are classified. By the above remark, 
if $(W,S,L)$ is a weighted Coxeter group such that $(W,S)$ is irreducible 
of finite type then $L$ takes at most $2$ different values on $S$. 

\begin{definition}
Let $\A=\Z[v,v^{-1}]$. Given weighted Coxeter group $(W,S,L)$, 
we define the associated Hecke algebra $\H(W,S,L)$, by 
generators $T_s$, for $s\in S$, and relations
$(T_s-v^{L(s)})(T_s+v^{-L(s)})=0$ and 
\begin{equation*}
\begin{split}
(T_sT_t)^{m_{st}/2}&=(T_tT_s)^{m_{st}/2} \;\text{if $m_{st}$ is even,} \\
(T_sT_t)^{(m_{st}-1)/2}T_s&=(T_tT_s)^{(m_{st}-1)/2}T_t\;\text{if $m_{st}$ is odd.}
\end{split}
\end{equation*}
\end{definition}

\begin{remark}
We may define multiparameter Hecke algebras when some of $m_{st}$ are even and not 
equal to $2$, 
but to handle the modular representation theory of Hecke algebras of finite type, 
the above definition suffices. However, we also note that the definition is for general 
weighted Coxeter groups, and affine cases are very interesting examples 
which we do not cover in this article. 
\end{remark}

Define $T_w=T_{s_{i_1}}\cdots T_{s_{i_{\ell(w)}}}$ when $w=s_{i_1}\cdots s_{i_{\ell(w)}}$, 
for $s_{i_1},\dots,s_{i_{\ell(w)}}\in S$. It is well-known that $T_w$ does not depend on 
the choice of the reduced expression. 

\begin{definition}
$\H(W,S,L)$ has an involutive $\A$-semilinear automorphism defined by 
$$
\overline{c(v)T_w}=c(v^{-1})T_{w^{-1}}^{-1},\quad(w\in W,\;c(v)\in\A).
$$
\end{definition}

\begin{theorem}[Kazhdan-Lusztig]
\label{KL thm}
Let $L_\infty=\oplus_{y\in W}\Z[v^{-1}]T_y$. 
For each $w\in W$, there exists a unique element $C_w\in L_\infty$ such that 
$$
\overline{C_w}=C_w\quad\text{and}\quad C_w\equiv T_w\!\!\mod v^{-1}L_\infty.
$$
Further, $\{C_w\mid w\in W\}$ is a free $\A$-basis of $\H(W,S,L)$. 
\end{theorem}

\begin{example}
Note that $\overline{T_s}=T_s-(v^{L(s)}-v^{-L(s)})$. 
\begin{enumerate}
\item[(i)]
If $s\in S$ then
$$
C_s=\begin{cases} T_s+v^{-L(s)}\quad&(L(s)>0)\\
T_s-v^{L(s)}\quad&(L(s)<0)\\
T_s\quad&(L(s)=0).\end{cases}
$$
\item[(ii)]
Assume that $L(s)>0$, for all $s\in S$, and that $W$ is finite. 
We denote the longest element of $W$ by $w_0$. Then 
$$
C_{w_0}=\sum_{y\in W}v^{-L(yw_0)}T_y.
$$
\end{enumerate}
\end{example}

The basis is called the \emph{canonical basis} of $\H(W,S,L)$ or 
the \emph{Kazhdan-Lusztig basis} (of second type). We write
$$
C_w=\sum_{y\in W}p_{y,w}T_y.
$$

\begin{remark}
Let $\mathcal K=\mathbb Q(v)$, 
$\A_0=\{c(v)\in\mathcal K\mid \text{$c(v)$ is regular at $v=0$}\}$ 
and $\A_\infty=\{c(v)\in\mathcal K\mid \text{$c(v)$ is regular at $v=\infty$}\}$. 
A $\mathcal K$-vector space $V$ is \emph{balanced} if there 
exist $\mathbb Q[v,v^{-1}]$-lattice $L$ of $V$, $\A_0$-lattice $\L_0$ and 
$\A_\infty$-lattice $\L_\infty$ of $V$ such that 
$E=L\cap \L_0\cap \L_\infty$ satisfies the following three properties.
\begin{enumerate}
\item[(1)]
Any $\mathbb Q$-basis of $E$ is a free $\mathbb Q[v,v^{-1}]$-basis of $L$. 
\item[(2)]
Any $\mathbb Q$-basis of $E$ is a free $\A_0$-basis of $\L_0$.
\item[(3)]
Any $\mathbb Q$-basis of $E$ is a free $\A_\infty$-basis of $\L_\infty$.
\end{enumerate}
It is easy to see that if $V$ is balanced then we have a canonical isomorphism 
of $\mathbb Q$-vector spaces 
$G:\L_\infty/v^{-1}\L_\infty\simeq E$. The Kazhdan-Lusztig theorem says that 
$\H(W,S,L)\otimes\mathcal K$ is balanced, and the basis 
$$
\{C_w:=G(T_w+v^{-1}\L_\infty)\mid w\in W\}
$$
is not only $\mathbb Q[v,v^{-1}]$-basis but $\A$-basis of $\H(W,S,L)$. 
\end{remark}

\begin{definition}
We write $C_xC_y=\sum_{z\in W}h_{x,y,z}C_z$, where $h_{x,y,z}\in\A$. 
Define $a(z)$, for $z\in W$, by
$$
a(z)=\min\{i\in\Z_{\geq0}\cup\{\infty\}\mid
v^ih_{x,y,z}\in\Z[v],\;\text{for all $x,y\in W$}\}.
$$
\end{definition}

This is Lusztig's $a$-function. In this subsection, we explain basics of 
the Kazhdan-Lusztig basis and the $a$-function following Lusztig and Geck.

Kazhdan and Lusztig proved Theorem \ref{KL thm} by inductively defining
$C_w$. Hence, they also showed the following theorem at the same time.

\begin{theorem}
\begin{enumerate}
\item[(1)]
Suppose that $L(s)=0$. Then $C_sC_w=C_{sw}$. 
\item[(2)]
Suppose that $L(s)>0$. Then
$$
C_sC_w=\begin{cases} C_{sw}+\sum_{z;sz<z<w}\mu^s_{z,w}C_z\;\;&(sw>w)\\
(v^{L(s)}+v^{-L(s)})C_w&(sw<w)\end{cases}
$$
where $\mu^s_{z,w}\in\A$ are bar invariant elements inductively defined by
$$
\sum_{z;y\leq z<w,sz<z}p_{y,z}\mu_{z,w}^s-v^{L(s)}p_{y,w}\in v^{-1}\Z[v^{-1}],
$$
for $y,w\in W$ such that $sy<y<w<sw$. 
\item[(3)]
Suppose that $L(s)<0$. Then
$$
C_sC_w=\begin{cases} C_{sw}+\sum_{z;sz<z<w}\mu^s_{z,w}C_z\;\;&(sw>w)\\
-(v^{L(s)}+v^{-L(s)})C_w&(sw<w)\end{cases}
$$
where $\mu^s_{z,w}\in\A$ are bar invariant elements inductively defined by
$$
\sum_{z;y\leq z<w,sz<z}p_{y,z}\mu_{z,w}^s+v^{-L(s)}p_{y,w}\in v^{-1}\Z[v^{-1}],
$$
for $y,w\in W$ such that $sy<y<w<sw$. 
\end{enumerate}
\end{theorem}

$\H(W,S,L)$ has an $\A$-linear antiautomorphism $\tau$ defined by 
$\tau(T_s)=T_s$, for $s\in S$. Then, $\tau(T_w)=T_{w^{-1}}$ and $\tau(C_w)=C_{w^{-1}}$, 
thus $h_{x,y,z}=h_{y^{-1},x^{-1},z^{-1}}$ follows. In particular, we have 
$a(z)=a(z^{-1})$, for all $z\in W$. 
The next proposition is from \cite[8.4, 13.7, 13.8]{L}.

\begin{proposition}
Suppose that $L(s)>0$, for all $s\in S$. 
\begin{enumerate}
\item[(1)] 
$\oplus_{w;ws<w}\A C_w$ and $\oplus_{w;sw<w}\A C_w$ are left and right ideal of 
$\H(W,S,L)$, and $a(1)=0$ follows.  
\item[(2)]
If $1\neq z\in W$ then $a(z)\geq\min\{L(s)\mid s\in S\}>0$.
\item[(3)]
Suppose that $W$ is finite. Then
$a(w)\leq L(w_0)$, for all $w\in W$, and the equality holds if and only if 
$w=w_0$. 
\end{enumerate}
\end{proposition}

\begin{example}
Let $(W,S)$ be of type $B_2$, and set $a=L(s_1), b=L(s_2)$. We consider the case  
$a>b>0$. Write $T_1=T_{s_1}$ and $T_2=T_{s_2}$. Then
$$
C_{s_1}=T_1+v^{-a},\;\;C_{s_2}=T_2+v^{-b}.
$$
Further, nonexistence of $z$ with $s_1z<z<s_2$ implies that 
$$
C_{s_1s_2}=C_{s_1}C_{s_2}=T_1T_2+v^{-b}T_1+v^{-a}T_2+v^{-a-b}.
$$
Similarly, we have
$$
C_{s_2s_1}=C_{s_2}C_{s_1}=T_2T_1+v^{-b}T_1+v^{-a}T_2+v^{-a-b}.
$$
Explicit computation shows that 
$$
C_{s_1}C_{s_2s_1}=T_1T_2T_1+v^{-a}(T_1T_2+T_2T_1)+v^{-2a}T_2+
(v^{a-b}+v^{-a-b})C_{s_1}.
$$
Since $a-b>0$ we subtract $(v^{a-b}+v^{-a+b})C_{s_1}$ to obtain
$$
C_{s_1s_2s_1}=T_1T_2T_1+v^{-a}(T_1T_2+T_2T_1)+v^{-2a}T_2+
(v^{-a-b}-v^{-a+b})C_{s_1}
$$
and $C_{s_1}C_{s_2s_1}=C_{s_1s_2s_1}+(v^{a-b}+v^{-a+b})C_{s_1}$. 
Similarly, we have
$$
C_{s_2}C_{s_1s_2}=T_2T_1T_2+v^{-b}(T_1T_2+T_2T_1)+v^{-2b}T_1+
(v^{-a+b}+v^{-a-b})C_{s_2}
$$
and deduce $C_{s_2}C_{s_1s_2}=C_{s_2s_1s_2}$. 

Now, consider the longest element $w_0=s_1s_2s_1s_2=s_2s_1s_2s_1$. We have
\begin{gather*}
C_{w_0}=T_1T_2T_1T_2+v^{-b}T_1T_2T_1+v^{-a}T_2T_1T_2+v^{-a-b}(T_1T_2+T_2T_1)\\
\quad+v^{-a-2b}T_1+v^{-2a-b}T_2+v^{-2a-2b}.
\end{gather*}
Then we get $C_{s_2}C_{s_1s_2s_1}=C_{w_0}$ by explicit computation again. 
Applying $\tau$ we obtain $C_{s_1s_2s_1}C_{s_2}=C_{w_0}$. Thus 
$C_{s_1}C_{s_2s_1s_2}=(C_{s_1}C_{s_2})^2$ implies
\begin{equation*}
\begin{split}
C_{s_1}C_{s_2s_1s_2}&=C_{s_1s_2s_1}C_{s_2}+(v^{a-b}+v^{-a+b})C_{s_1}C_{s_2}\\
&=C_{w_0}+(v^{a-b}+v^{-a+b})C_{s_1s_2}.
\end{split}
\end{equation*}
To summarize, $C_sC_{w_0}=(v^{L(s)}+v^{-L(s)})C_{w_0}$, for $s\in S$, and
we have
\begin{gather*}
C_{s_1}C_{s_2}=C_{s_1s_2},\;C_{s_2}C_{s_1}=C_{s_2s_1},\;
C_{s_2}C_{s_1s_2}=C_{s_2s_1s_2}\\
C_{s_1}C_{s_2s_1}=C_{s_1s_2s_1}+(v^{a-b}+v^{-a+b})C_{s_1}
\end{gather*}
and
$$
C_{s_2}C_{s_1s_2s_1}=C_{w_0},\;
C_{s_1}C_{s_2s_1s_2}=C_{w_0}+(v^{a-b}+v^{-a+b})C_{s_1s_2}.
$$
We have also obtained $C_{s_1s_2}^2=C_{w_0}+(v^{a-b}+v^{-a+b})C_{s_1s_2}$. 
Using these formulas, we may further obtain the following.
\begin{align*}
C_{s_1s_2}C_{s_2s_1}&=(v^b+v^{-b})C_{s_1s_2s_1}+(v^b+v^{-b})(v^{a-b}+v^{-a+b})C_{s_1}\\
C_{s_1s_2}C_{s_1s_2s_1}&=(v^a+v^{-a})C_{w_0}\\
C_{s_1s_2}C_{s_2s_1s_2}&=(v^b+v^{-b})C_{w_0}+(v^b+v^{-b})(v^{a-b}+v^{-a+b})C_{s_1s_2}\\
C_{s_2s_1}^2&=C_{w_0}+(v^{a-b}+v^{-a+b})C_{s_2s_1}\\
C_{s_2s_1}C_{s_1s_2s_1}&=(v^a+v^{-a})C_{w_0}\\
C_{s_2s_1}C_{s_2s_1s_2}&=(v^b+v^{-b})C_{w_0}+(v^{a-b}+v^{-a+b})C_{s_2s_1s_2}\\
C_{s_1s_2s_1}^2&=(v^a+v^{-a})^2C_{w_0}-(v^a+v^{-a})(v^{a-b}+v^{-a+b})C_{s_1s_2s_1}\\
C_{s_1s_2s_1}C_{s_2s_1s_2}&=(v^a+v^{-a})(v^b+v^{-b})C_{w_0}\\
C_{s_2s_1s_2}^2&=(v^b+v^{-b})^2C_{w_0}+(v^b+v^{-b})(v^{a-b}+v^{-a+b})C_{s_2s_1s_2}
\end{align*}

Let $\Gamma_1=\{1\}$, $\Gamma_2=\{s_1,s_2s_1\}$, 
$\Gamma_3=\{s_2\}$, $\Gamma_4=\{s_1s_2,s_2s_1s_2\}$, $\Gamma_5=\{s_1s_2s_1\}$, $\Gamma_6=\{w_0\}$. 
We define $\mathcal I_{\leq1}=\H(W,S,L)$ and 
\begin{gather*}
\mathcal I_{\leq2}=\sum_{w\in \Gamma_2\sqcup\Gamma_5\sqcup\Gamma_6}\A C_w,\quad
\mathcal I_{\leq3}=\sum_{w\in \Gamma_3\sqcup\Gamma_4\sqcup\Gamma_6}\A C_w\\
\mathcal I_{\leq4}=\sum_{w\in \Gamma_4\sqcup\Gamma_6}\A C_w,\quad
\mathcal I_{\leq5}=\sum_{w\in \Gamma_5\sqcup\Gamma_6}\A C_w,\quad
\mathcal I_{\leq6}=\sum_{w\in \Gamma_6}\A C_w.
\end{gather*}
These are left ideals. Further,
$$
\H(W,S,L)=\mathcal I_{\leq1}\supseteq
\mathcal I_{\leq2}+\mathcal I_{\leq3}\supseteq
\mathcal I_{\leq2}+\mathcal I_{\leq4}\supseteq
\mathcal I_{\leq5}+\mathcal I_{\leq6}\supseteq
\mathcal I_{\leq6}
$$
is a filtration by two-sided ideals. The sets 
$\Gamma_1, \Gamma_2\sqcup\Gamma_4, \Gamma_3, \Gamma_5, \Gamma_6$ are called two-sided cells, 
and we may confirm that the $a$-function is constant on each two-sided cell. 
Namely, $a(z)=0$ if $z\in\Gamma_1$, $a(z)=a$ if $z\in\Gamma_2\sqcup\Gamma_4$, 
$a(z)=b$ if $z\in\Gamma_3$, $a(z)=2a-b$ if $z\in\Gamma_5$, $a(z)=2a+2b$ if $z\in\Gamma_6$. 
This is a general phenomenon. 
\end{example}

In the rest of the paper, we assume that
\begin{center}
{\bf $W$ is finite and $L(s)>0$, for all $s\in S$}. 
\end{center}

\begin{definition}
Let $R$ be a commutative domain, $\A \rightarrow R$ a ring homomorphism. Then 
we define $\H_R=\H(W,S,L)\otimes R$. We denote the image of $v\in\A$ by $q^{1/2}\in R$.
\end{definition}

It is known that $\H_R$ may be defined by the same defining relations as $\H(W,S,L)$. 

$\H_R$ has a trace map $\tr: \H\rightarrow \A$ 
defined by $\tr(T_w)=0$ if $w\neq1$, $\tr(T_1)=1$. For the symmetric bilinear form 
defined by $\langle h,h'\rangle=\tr(hh')$, we have 
$\langle T_{y^{-1}},T_w\rangle=\delta_{yw}$. Hence, if $R$ is a field then 
$\H_R$ is a symmetric algebra. 

Let $\mathcal K=\mathbb Q(v)$ as before. Then $\H_{\mathcal K}$ is 
split semisimple. The simple $\H_{\mathcal K}$-modules are in bijection with 
$\Irr(W)$, and we denote by $\{V^E\mid E\in\Irr(W)\}$ the complete set of 
simple $\H_{\mathcal K}$-modules. Then 
$$
\sum_{w\in W}\mattr(T_{w^{-1}},V^E)T_w
$$
is a central element of $\H_{\mathcal K}$, which acts on $V^E$ by the scalar 
$$
c_E=\frac{1}{\dim E}\sum_{w\in W}\mattr(T_{w^{-1}},V^E)\mattr(T_w,V^E).
$$

Observe that $\A$ is integrally closed in $\mathcal K$. The following is
proved in \cite[7.3.8]{GP}.

\begin{proposition}
$\mattr(T_w,V^E)\in\A$, for $w\in W$ and $E\in\Irr(W)$. 
\end{proposition}

\begin{definition}
The \emph{$a$-invariant} $a_E$, for $E\in\Irr(W)$, is defined by
$$
a_E=\min\{i\in\Z_{\ge0}\mid v^i\mattr(T_w,V^E)\in\Z[v],\;
\text{for all $w\in W$.}\}
$$
\end{definition}

\begin{proposition}[Lusztig]
We may write 
$c_E=f_Ev^{-2a_E}+\text{(higher terms)}$, for some $f_E\in\Z_{>0}$. 
\end{proposition}

$c_E$ are called the \emph{Schur elements} and we have 
$$
\tr=\sum_{E\in\Irr(W)}\frac{1}{c_E}\mattr(-,V^E).
$$
This result suggests that we may define $a_E$ for more general $\A$-algebras that 
has a trace map. In fact, Jacon developed 
a theory of $a$-invariants for the cyclotomic Hecke algebra of type $(d,1,n)$, 
which is also called the AK-algebra. 

\begin{definition}
A field $R$ is \emph{$L$-good} if $f_E$ is invertible in $R$, for all $E\in\Irr(W)$.
\end{definition}

The irreducible characters of generic Hecke algebras are explicitly known. Hence 
we may compute $f_E$ by substituting the parameters with $v^{L(s)}$ and expand 
$c_E$ into the Laurent series in $v$. 
When $L$ is the length function then we have ordinary notion of good primes.
The following result of Geck shows that most primes are $L$-good.

\begin{lemma}
Suppose that $L(s)>0$, for all $s\in S$. If the characteristic of $R$ is a good prime 
then $R$ is $L$-good. 
\end{lemma}

In particular, if the characteristic of $R$ is different from $2, 3, 5$ then $R$ is 
$L$-good. The following example, which is called the asymptotic case,
was studied by \cite{BI}, \cite{GI} and \cite{GIP}.

\begin{example}
Let $W$ be the Weyl group of type $B_n$. The Coxeter generators are denoted by 
$s_0,\dots,s_{n-1}$ such that $s_1$,\dots,$s_{n-1}$ generate 
the symmetric group of degree $n$. Write $L(s_0)=b, L(s_1)=\cdots=L(s_{n-1})=a$ and 
suppose that $b>(n-1)a>0$. Then $f_E=1$, for all $E\in\Irr(W)$, and all fields are 
$L$-good. On the other hand, $2$ is a bad prime. 
\end{example}

\subsection{Cellularity}

\begin{definition}
Let $R$ be a commutative domain, $A$ an $R$-algebra. $A$ is \emph{cellular} if 
there exist a finite poset $\Lambda$, a collection of finite sets 
$\{M(\lambda)\mid \lambda\in\Lambda\}$ and an $R$-linear antiautomorphism of $A$ 
which is denoted by $a\mapsto a^*$, for $a\in A$, such that 
\begin{enumerate}
\item[(i)]
$A$ has a free $R$-basis $\sqcup_{\lambda\in\Lambda}\{C_{ST}^\lambda\mid S,T\in M(\lambda)\}$.
\item[(ii)]
$(C_{ST}^\lambda)^*=C_{TS}^\lambda$. 
\item[(iii)]
$A^{\geq\lambda}=\sum_{\mu\geq\lambda}\sum_{S,T\in M(\mu)}RC_{ST}^\mu$ is a two-sided ideal of $A$, for 
all $\lambda\in\Lambda$. 
\item[(iv)]
For each $h\in A$, $S,T\in M(\lambda)$, there exist $r(h,S,T)\in R$ such that we have 
$$
hC_{SU}^\lambda\equiv \sum_{T\in M(\lambda)}r(h,S,T)C_{TU}^\lambda\mod A^{>\lambda},
$$
for all $U\in M(\lambda)$. 
\end{enumerate}
\end{definition}

K\"onig and Changchang Xi \cite{KX1} showed that an $R$-algebra is cellular if and only if 
it is obtained from a particular construction which is called the iterated 
inflation of finitely many copies of $R$. 

\begin{definition}
Assume that an $R$-algebra $A$ is cellular. Define an $A$-module 
$C^\lambda=\oplus_{S\in M(\lambda)}RC_S^\lambda$ by
$$
hC_S^\lambda=\sum_{T\in M(\lambda)}r(h,S,T)C_T^\lambda,\quad\text{for $h\in A$.}
$$
$C^\lambda$ is called a \emph{cell module}. $C^\lambda$ is equipped with a bilinear from defined by 
$$
C_{US}^\lambda C_{TV}^\lambda\equiv \langle C_S^\lambda, C_T^\lambda\rangle C_{UV}^\lambda\mod A^{>\lambda}.
$$
The radical $\Rad_{\langle\hphantom{-},\hphantom{-}\rangle}C^\lambda$ of the bilinear form 
is an $A$-submodule. Define $D^\lambda$ by
$$
D^\lambda=C^\lambda/\Rad_{\langle\hphantom{-},\hphantom{-}\rangle}C^\lambda.
$$
\end{definition}

The following are basic results on cellular algebras. See \cite{GL}, \cite{KX2}
and \cite{XX}.

\begin{theorem}[Graham-Lehrer]
Let $R$ be a field, $A$ a cellular $R$-algebra. 
\begin{enumerate}
\item[(i)]
Nonzero $D^\lambda$'s form a complete set of simple $A$-modules.
Further, if $D^\lambda\neq0$ then it is absolutely irreducible and
the Jacobson radical $\Rad C^\lambda$ coincides with
$\Rad_{\langle\hphantom{-},\hphantom{-}\rangle}C^\lambda$.
In particular,
$\{C^\lambda\mid \lambda\in\Lambda\}$ is a complete set of simple $A$-modules if $A$ is semisimple.
\item[(ii)]
If $D^\lambda\neq0$, for all $\lambda\in\Lambda$, then $A$ is quasihereditary.
In particular, $A$ has finite global dimension in this case.
\item[(iii)]
Let $\Lambda^o=\{\lambda\in\Lambda\mid D^\lambda\neq0\}$,
$D=([C^\lambda:D^\mu])_{(\lambda,\mu)\in\Lambda\times\Lambda^o}$
the decomposition matrix.
Then $D$ is unitriangular: $d_{\lambda\mu}\neq0$
only if $\lambda\geq\mu$ and $d_{\mu\mu}=1$.
\item[(iv)]
The Cartan matrix of $A$ is of the form $C=\hphantom{}^tDD$.
\end{enumerate}
\end{theorem}

\begin{theorem}[K\"onig-Xi]
Let $R$ be a field, $A$ a cellular $R$-algebra.
\begin{enumerate}
\item[(i)]
If $A$ is self-injective then it is weakly symmetric, that is, the head and
the socle of any indecomposable projective $A$-module are isomorphic.
\item[(ii)]
Assume that the characteristic of $R$ is odd. If another $R$-algebra $B$ is
Morita equivalent to $A$ then $B$ is cellular.
\end{enumerate}
\end{theorem}

Recall that the trivial extension $T(A)=A\oplus\Hom_R(A,R)$ is a symmetric algebra 
whose bilinear form is given by $\langle a\oplus f,b\oplus g\rangle=f(b)+g(a)$, 
for $a,b\in A$ and $f,g\in\Hom_R(A,R)$. 

\begin{theorem}[Xi-Xiang]
Let $R$ be a field, $A$ a cellular $R$-algebra. Then we may define 
an antiautomorphism of $T(A)$ by $(a\oplus f)^*=a^*\oplus f^*$, 
where $f^*(b)=f(b^*)$, for $b\in A$, such that $T(A)$ is a cellular 
$R$-algebra. In particular, any cellular $R$-algebra is a quotient of 
a symmetric cellular $R$-algebra. 
\end{theorem}

Now we return to Hecke algebras and state a result by Geck \cite{G2}. 
In fact, it is proved 
under more general assumption that part of the Lusztig conjectures namely 
(P2)-(P8) and (P15') \cite[5.2]{G1} hold. These are conjectures are about 
the structure constants $h_{x,y,z}$, and do not involve the base ring $R$. 
It is known that 
the conjectures hold when $L$ is the length function. In this case, 
$\H(W,S,L)^{\geq a}=\sum_{w\in W, a(w)\geq a}\A C_w$ is a two-sided ideal of 
$\H(W,S,L)$, for all $a\in\Z$. Then each successive quotient is a direct sum of 
Lusztig's two-sided cells. By refining the ideal filtration, Geck has proved 
the following. Recall that $\H_R=\H(W,S,L)\otimes R$.

\begin{theorem}[Geck]
\label{Geck cellular}
Let $(W,S,L=\ell)$ be a weighted Coxeter group of finite type whose 
$L$ is the length function. Suppose that $R$ is $L$-good. Then 
$\H_R$ is cellular. 
\end{theorem}

This opens a way to consider the possibilities to find analogues of many 
results that appeared 
in Specht module theory. For example, studying Young modules in
this setting is important.  

\subsection{Canonical basic set}

The first task in studying the modular representation theory of
$\H_R$ is to determine the set of simple $\H_R$-modules.
This very first task has already proven to be very interesting
and deep.

Let $R$ be a field, $\A\rightarrow R$ an algebra homomorphism.
Its kernel is a prime ideal $\mathfrak p$ of $\A$ and
we may consider modular reduction between $\H_{\A_{\mathfrak p}}$
and $\H_R$. For a simple $\H_R$-module, we denote by
$[E:S]$ the multiplicity of $S$ in the modular reduction of $V^E$.

\begin{definition}
For a simple $\H_R$-module $S$, the \emph{$a$-invariant} of $S$ is defined by
$$
a_S=\min\{a_E\mid E\in\Irr(W)\;\text{such that $[E:S]\neq0$.}\}
$$
\end{definition}

Geck and Rouquier used the $a$-invariant to label simple $\H_R$-modules.

\begin{definition}
We say that a subset $B$ of $\Irr(W)$ is a \emph{canonical basic set} if
\begin{enumerate}
\item[(i)]
There is a bijection $B\simeq\Irr(\H_R)$, which we denote $E\mapsto S^E$,
such that $[E:S^E]=1$ and $a_{S^E}=a_E$.
\item[(ii)]
If a simple $\H_R$-module $S$ satisfies $[E:S]\neq0$, for some $E\in\Irr(W)$,
then either $E\in B$ and $S\simeq S^E$ or $a_S<a_E$.
\end{enumerate}
\end{definition}

If a canonical basic set exists, then $E\not\in B$ implies that
$a_S<a_E$, for all $S$ with $[E:S]\neq0$, and we have
$$
B=\{E\in\Irr(W)\mid \text{$a_E=a_S$ and $[E:S]\neq0$, for some $S\in\Irr(\H_R)$.}\}.
$$
As the right hand side is independent of the choice of $B$, we have the uniqueness 
of the canonical basic set. Note however that it may not exist. Geck showed that 
under the assumption that the Lusztig conjectures hold, the canonical basic set
exists if $R$ is $L$-good. In particular, it implies the following result,
which was proved in the early stage of their research.

\begin{theorem}[Geck-Rouquier]
Suppose that $L$ is the length function and that $R$ is $L$-good.
Then the canonical basic set exists. 
\end{theorem}

In fact, it is a corollary of Theorem \ref{Geck cellular} and
the canonical basic set is nothing but the index set of
simple $\H_R$-modules given by the cellular algebra structure.
In type $B_n$ we have a result for arbitrary $L$ by Geck and Jacon \cite{GJ}.

\begin{theorem}[Geck-Jacon]
Let $\H_R=\H(W,S,L)\otimes R$ be of type $B_n$. Assume that
the characteristic of $R$ is $0$. Then the canonical basic set exists.
\end{theorem}

The existence is still conjectural in positive characteristics unless $L$
is (a positive multiple of) the length function.

\subsection{A categorification of integrable $U_v(\hat{sl}_e)$-modules}

The canonical basic set is determined by Geck and Jacon for all types.
Let us focus on Hecke algebras of type $B_n$. As we will explain
in the next section for more general cyclotomic Hecke algebras,
the author and Mathas used Kashiwara crystal for
labelling simple $\H_R$-modules. The set of bipartitions that
appeared in this labelling, which we call Kleshchev bipartitions, is a realization
of the highest weight crystal $B(\Lambda)$ whose highest weight $\Lambda$ is
determined by the parameters of $\H_R$.
On the other hand, there was a different realization of the crystal $B(\Lambda)$
by Jimbo-Misra-Miwa-Okado in a purely solvable lattice model context,
and Foda, Leclerc, Okado, Thibon and Welsh proposed
another labelling of simple $\H_R$-modules that uses it. Then, Jacon found that
the canonical basic set in type $B_n$ is precisely the set of
Jimbo-Misra-Miwa-Okado bipartitions.
Hence, the set of JMMO bipartitions gives module theoretic realization of the
labeling of simple $\H_R$-modules proposed by Foda et al.

In the labelling by Kleshchev bipartitions, we used Specht module theory which
was developed by Dipper, James and Murphy; they gave a cellular algebra structure on
$\H_R=\H(W(B_n),S,L)\otimes R$, for any $L$. We showed that
the labelling of simple $\H_R$-modules by Kleshchev bipartitions is nothing but
the labelling induced by the cellular structure.
Difference of Kleshchev and JMMO bipartitions is caused by
the difference of the orders given on the nodes of bipartitions, but
we seek for explanations in the representation theory of
Hecke algebras why two (or more) different labelling sets appear naturally.
The key seems to be various choices of the logarithm of
the parameters of $\H_R$.
It is now conjectured \cite{BGIL} that the choice 
would give a cellular algebra structure on $\H_R$ which is given by
Kazhdan-Lusztig cells,
and a parametrizing set of simple $\H_R$-modules, which we call Uglov bipartitions,
as the one induced by the cellular algebra structure.
There are two supporting evidences.
Geck, Iancu and Pallikaros \cite{GIP} showed that our labelling by Kleshchev bipartitions
may be considered as a special case of this scheme, and Geck and Jacon \cite{GJ} showed
that the set of Uglov bipartitions is the canonical basic set, for any $L$, when
the characteristic of $R$ is $0$.

This search for various cellular algebra structures may be viewed as a search for
a categorification of integrable highest weight
$U_v(\hat{sl}_e)$-modules $V_v(\Lambda)$ with two specializations at $v=0$ and
$v=1$. Here, by specialization at $v=0$ we mean the crystal $B(\Lambda)$, and
specialization at $v=1$ we mean the integrable highest weight $\hat{sl}_e$-module $V(\Lambda)$.
Let $\mathcal F(\Lambda)$ be the higher level Fock space with highest weight $\Lambda$.
It is the tensor product of $\mathcal F(\Lambda_m)$ which will be introduced in 3.2
below and the basis is given by multipartitions. Then we have the following diagram.
\begin{center}
$B(\Lambda)$ $\stackrel{v=0}{\longleftarrow}$
$V_v(\Lambda)$ $\stackrel{v=1}{\longrightarrow}$ $V(\Lambda)\subseteq\mathcal F(\Lambda)$.
\end{center}
We fix an embedding of $V_v(\Lambda)$ into one of various JMMO deformed Fock spaces
$\mathcal F_v(\Lambda)$\footnote{They are not tensor product of $\mathcal F_v(\Lambda_m)$.}
in the middle, and realize $B(\Lambda)$ on the set of
Uglov multipartitions on the left, and we categorify them. Then
our categorification is to replace each weight space $V_v(\Lambda)_\mu$ of $V_v(\Lambda)$ with
module category of a cellular algebra $A_\mu$ whose poset is the set of multipartitions
which belong to $\mathcal F(\Lambda)_\mu$ such that
\begin{itemize}
\item[(i)]
the set $B(\Lambda)_\mu$ of Uglov multipartitions on the left coincides
the parametrizing set of $\Irr(A_\mu)$ induced by the cellular algebra structure
on $A_\mu$,
\item[(ii)]
$V(\Lambda)_\mu$ on the right coincides $\Hom_\Z(K_0(A_\mu\text{-}mod),\C)$,
\item[(iii)]
the embedding $V(\Lambda)_\mu\subseteq\mathcal F(\Lambda)_\mu$ coincides the dual of the
decomposition map,
\item[(iv)]
the Chevalley generators $e_i, f_i$ lift to functors among the module categories.
\end{itemize}
Our candidates for $A_\mu$ are block algebras of cyclotomic Hecke algebras.

We expect to categorify the higher level Fock space $\mathcal F(\Lambda)$ itself by
Rouquier's theory of quasihereditary covers which uses 
rational Cherednik algebras associated with the complex reflection group
$G(d,1,n)$. As we will explain in the third part,
we have the category $\mathcal O$ of the rational Cherednik algebra, which
is equivalent to the module category of a quasihereditary algebra, and
the KZ functor from $\mathcal O$
to the module category of the Hecke algebra. We expect to have categorification 
such that (i) to (iv) and
\begin{itemize}
\item[(v)]
the quasihereditary structure and the cellular structure
induce the embedding of the set of Uglov multipartitions into the set of
multipartitions,
\item[(vi)]
$\mathcal F(\Lambda)$ coincides $\Hom_\Z(K_0(\mathcal O),\C)$,
\item[(vii)]
the dual of the KZ functor coincides the dual of the decomposition map,
\item[(viii)]
the dual basis of $\Irr(\mathcal O)$ coincides the Uglov canonical basis at $v=1$.
\end{itemize}
The last part is Yvonne's conjecture.
We said that the logarithm of the parameters of the Hecke algebra seems to
control the cellular structure. Here, the logarithm appears as the
parameters of the rational Cherednik algebra, and what we expect in (v) is
that the quasihereditary algebra structure should induce the cellular
algebra structure on the Hecke algebra.

\section{The category of crystals}
\subsection{Kashiwara crystal}

Let us recall the definition. 

\begin{definition}
Let $A=(a_{ij})_{i,j\in I}$ be a generalized Cartan matrix, 
$$
(A,\Pi=\{\alpha_i\}_{i\in I},\Pi^{\vee}=\{h_i\}_{i\in I},P,P^{\vee}=\Hom_{\Z}(P,\Z))
$$
a root datum. 
Let $\mathfrak g=\mathfrak g(A)$ be the Kac-Moody algebra associated with $A$. 
A set $B$ is a \emph{$\mathfrak g$-crystal} if it is equipped with  
maps $\wt: B\rightarrow P$, $\tilde e_i,\tilde f_i:B\rightarrow B\sqcup\{0\}$, 
$\epsilon_i,\varphi_i:B\rightarrow\Z\sqcup\{-\infty\}$ such that 
\begin{enumerate}
\item[(1)]
$\varphi_i(b)=\epsilon_i(b)+\langle h_i,\wt(b)\rangle$.
\item[(2)]
If $\tilde e_ib\in B$ then
$$
\wt(\tilde e_ib)=\wt(b)+\alpha_i,\;\epsilon_i(\tilde e_ib)=\epsilon_i(b)-1,\;
\varphi_i(\tilde e_ib)=\varphi_i(b)+1.
$$
\item[(3)]
If $\tilde f_ib\in B$ then
$$
\wt(\tilde f_ib)=\wt(b)-\alpha_i,\;\epsilon_i(\tilde f_ib)=\epsilon_i(b)+1,\;
\varphi_i(\tilde e_ib)=\varphi_i(b)-1.
$$
\item[(4)]
Let $b, b'\in B$. Then $\tilde f_ib=b'$ if and only if $\tilde e_ib'=b$. 
\item[(5)]
If $\varphi_i(b)=-\infty$ then $\tilde e_ib=0$ and $\tilde f_ib=0$.
\end{enumerate}
\end{definition}

$B$ may be viewed as an $I$-colored oriented graph by writing 
$b\overset{i}{\rightarrow}b'$ if $\tilde f_ib=b'$. We call this graph 
the \emph{crystal graph} of $B$. 

\begin{example}
\label{crystal examples}
\begin{enumerate}
\item[(1)]
Let $\mathfrak g=sl_2$, $\alpha$ the positive root, $\omega=\alpha/2$ the fundamental weight. 
Let $B(n\omega)=\{u_0,u_1,\dots,u_n\}$ and define 
$$
\wt(u_k)=n\omega-k\alpha,\;\epsilon(u_k)=k,\;\varphi(u_k)=n-k\;\;\text{and}
$$
$$
\tilde eu_k=\begin{cases}u_{k-1}&(k>0)\\0&(k=0)\end{cases}, \quad
\tilde fu_k=\begin{cases}u_{k+1}&(k<n)\\0&(k=n)\end{cases}.
$$
Next, let $B(\infty)=\{u_k\mid k\in\Z_{\geq0}\}$ and define
$$
\wt(u_k)=-k\alpha,\;\epsilon(u_k)=k,\;\varphi(u_k)=-k\;\;\text{and}
$$
$$
\tilde eu_k=\begin{cases}u_{k-1}&(k>0)\\0&(k=0)\end{cases}, \quad
\tilde fu_k=u_{k+1}.
$$
Then, $B(n\omega)$ and $B(\infty)$ are $\mathfrak g$-crystals. 
\item[(2)]
Let $B_i=\{b_i(a)\mid a\in\Z\}$. Define, for $a\in\Z$, 
$\wt(b_i(a))=a\alpha_i$,
$$
\epsilon_j(b_i(a))=\begin{cases} -a&(j=i)\\
                            -\infty&(j\neq i)\end{cases}, \quad
\varphi_j(b_i(a))=\begin{cases}  a&(j=i)\\
                            -\infty&(j\neq i)\end{cases}
$$
and
$$
\tilde e_j(b_i(a))=\begin{cases} b_i(a+1)&(j=i)\\
                            0&(j\neq i)\end{cases}, \quad
\tilde f_j(b_i(a))=\begin{cases}  b_i(a-1)&(j=i)\\
                             0&(j\ne i)\end{cases}.
$$
Then $B_i$ is a $\mathfrak g$-crystal.
\item[(3)]
Let $\Lambda\in P$ and $T_\Lambda=\{t_\Lambda\}$. Define
$$
\wt(t_\Lambda)=\Lambda,\;\;\epsilon_i(t_\Lambda)=\varphi_i(t_\Lambda)=-\infty,
\;\;
\tilde e_it_\Lambda=\tilde f_it_\Lambda=0.
$$
Then $T_\Lambda$ is a $\mathfrak g$-crystal.
\end{enumerate}
\end{example}

\begin{definition}
Let $B_1, B_2$ be $\mathfrak g$-crystals. A \emph{crystal morphism} 
is a map $$f:B_1\sqcup\{0\}\rightarrow B_2\sqcup\{0\}$$ such that 
\begin{itemize}
\item[(i)]
$f(0)=0$.
\item[(ii)]
Suppose that $b\in B_1$ and $f(b)\in B_2$. Then 
$$
\wt(f(b))=\wt(b),\;\epsilon_i(f(b))=\epsilon_i(b),\;\varphi_i(f(b))=\varphi_i(b).
$$
\item[(iii)]
Suppose that $b, b'\in B_1$ and $f(b), f(b')\in B_2$. If $\tilde f_ib=b'$ then 
$\tilde f_if(b)=f(b')$. 
\item[(iv)]
Suppose that $b, b'\in B_1$ and $f(b)=0, f(b')\neq0$. If $b=\tilde e_ib'$ 
(resp. $b=\tilde f_ib'$) then $\tilde e_if(b')=0$ (resp. $\tilde f_if(b')=0$). 
\end{itemize}
\end{definition}
If $f$ is injective, we say that $f$ is an \emph{embedding}. If $f$ is bijective then 
we say that $f$ is an \emph{isomorphism}. For example, the identity map is an isomorphism.

\begin{remark}
The definition of crystal morphism in \cite{HonKan}, \cite{Jo} and \cite{Kas}
are all different. In \cite{HonKan}, which follows \cite{K1}, (iv)
is dropped. In \cite{Kas}, $f$ is assumed to map $B_1$ to $B_2$.
Let us consider $B(2\omega)=\{u_0,u_1,u_2\}$ in
Example \ref{crystal examples}(1).
The map $f:B(2\omega)\sqcup\{0\}\rightarrow B(2\omega)\sqcup\{0\}$ defined
by $f(0)=0$, $f(u_i)=u_i$, for $i=0,1$, and $f(u_2)=0$ satisfies (i), (ii) and (iii) but
not (iv). As $B(2\omega)$ corresponds to the irreducible highest weight module
$V_v(2\omega)$, we would like that the identity map is the only crystal
endomorphism of $B(2\omega)$.
\end{remark}

Note that a crystal morphism $f$ may not commute with $\tilde e_i$ and $\tilde f_i$.
If $f$ commutes with them, we say that $f$ is a
\emph{strict crystal morphism}. The strictness further requires
\begin{enumerate}
\item[(v)]
Suppose that $\tilde e_ib=0$ (resp. $\tilde f_ib=0$). Then  
$\tilde e_if(b)=0$ (resp. $\tilde f_if(b)=0$). 
\end{enumerate}

\begin{example}
Let $B$ be a $\mathfrak g$-crystal. 
Define a new crystal $(B,\wt^\sigma,\tilde e_i^\sigma, \tilde f_i^\sigma, 
\epsilon_i^\sigma,\varphi_i^\sigma)$ by
\begin{gather*}
\wt^\sigma(b)=\wt(\sigma^{-1}(b)),\;\epsilon_i^\sigma(b)=\epsilon_i(\sigma^{-1}(b)),\;
\varphi_i^\sigma(b)=\varphi_i(\sigma^{-1}(b)),\\
\tilde e_i^\sigma b=\sigma\tilde e_i\sigma^{-1}(b),\;
\tilde f_i^\sigma b=\sigma\tilde f_i\sigma^{-1}(b),
\end{gather*}
where $\sigma:B\rightarrow B$ is a permutation. 
Then $f:B\sqcup\{0\}\rightarrow B\sqcup\{0\}$ defined by $f(0)=0$ and 
$f(b)=\sigma(b)$ $(b\in B)$ is an isomorphism, which is strict. 
Hence, if $B$ is given two crystal structures 
which are isomorphic, it does not mean that $\tilde e_i$ and $\tilde f_i$ of the two 
crystal structures coincide.  
\end{example}

\begin{example}
Let $\mathfrak g=sl_2$ and $(B(\infty),\wt,\tilde e,\tilde f,\epsilon,\varphi)$ as above. 
Define a new crystal
$$
B(\infty)\otimes T_{n\omega}=(B(\infty), \wt+n\omega, \tilde e, \tilde f, \epsilon, \varphi+n).
$$
\begin{enumerate}
\item[(1)]
The map $f:B(n\omega)\sqcup\{0\}\rightarrow B(\infty)\otimes T_{n\omega}\sqcup\{0\}$ defined by 
$f(0)=0$ and $f(u_k)=u_k$, for $0\leq k\leq n$, is a crystal morphism. However, 
$\tilde fu_n=0$ in $B(n\omega)$ and $\tilde fu_n=u_{n+1}\neq0$ in $B(\infty)\otimes T_{n\omega}$. 
Thus the morphism is not strict. 
\item[(2)]
The map $f:B(\infty)\otimes T_{n\omega}\sqcup\{0\}\rightarrow B(n\omega)\sqcup\{0\}$ defined by
$$
f(u_k)=\begin{cases}u_k &(k\leq n)\\0 &(k>n)\end{cases}
$$
is a strict crystal morphism. 
\end{enumerate}
\end{example}

Crystals and morphisms among them form a category, which are called
the category of $\mathfrak g$-crystals. We have the notion of tensor product 
in the category. 

\begin{definition}
Let $B_1, B_2$ be $\mathfrak g$-crystals. The tensor product $B_1\otimes B_2$ is 
the set $B_1\times B_2$ equipped with the crystal structure defined by
\begin{enumerate}
\item[(1)]
$\wt(b_1\otimes b_2)=\wt(b_1)+\wt(b_2)$.
\item[(2)]
$\tilde e_i(b_1\otimes b_2)=\tilde e_ib_1\otimes b_2$ if 
$\varphi_i(b_1)\geq\epsilon_i(b_2)$, $b_1\otimes\tilde e_ib_2$ otherwise. 
\item[(3)]
$\tilde f_i(b_1\otimes b_2)=\tilde f_ib_1\otimes b_2$ if 
$\varphi_i(b_1)>\epsilon_i(b_2)$, $b_1\otimes\tilde f_ib_2$ otherwise. 
\item[(4)]
$\epsilon_i(b_1\otimes b_2)=\max\{\epsilon_i(b_1),\epsilon_i(b_2)-\langle h_i,\wt(b_1)\rangle\}$.
\item[(5)]
$\varphi_i(b_1\otimes b_2)=\max\{\varphi_i(b_1)+\langle h_i,\wt(b_2)\rangle,\varphi_i(b_2)\}$.
\end{enumerate}
\end{definition}

\begin{example}
Let $B$ be a $\mathfrak g$-crystal, $\Lambda\in P$. Then 
$B\otimes T_\Lambda$ is a crystal with the same $\tilde e_i, \tilde f_i, 
\epsilon_i$ as $B$ but $\wt$ and $\varphi_i$ are shifted by $\Lambda$ and 
$\Lambda(h_i)$. 
\end{example}

Recall that a monoidal category $\mathcal C$ is a category with a 
bifunctor $\mathcal C\times\mathcal C\rightarrow\mathcal C$, an object $I$ called 
the unit object, such that natural isomorphisms 
\begin{gather*}
\alpha_{B_1B_2B_3}:(B_1\otimes B_2)\otimes B_3\simeq B_1\otimes(B_2\otimes B_3)\\
\lambda_B:I\otimes B\simeq B,\;\;\rho_B:B\otimes I\simeq B
\end{gather*}
satisfy axioms for $B_1\otimes B_2\otimes B_3\otimes B_4$, 
$B_1\otimes I\otimes B_2\simeq B_1\otimes B_2$ 
(the pentagon axiom and the triangle axiom) and 
$\lambda_I=\rho_I:I\otimes I\simeq I$. 

For a crystal morphism $f:B_1\rightarrow B_2$, we have crystal morphisms 
$B\otimes B_1\rightarrow B\otimes B_2$ and $B_1\otimes B\rightarrow B_2\otimes B$
given by $b\otimes b'\mapsto b\otimes f(b')$ and $b\otimes b'\mapsto f(b)\otimes b'$,  
and the tensor product defines a bifunctor. The identity map gives 
a natural isomorphism $(B_1\otimes B_2)\otimes B_3\simeq B_1\otimes(B_2\otimes B_3)$. 
We have natural isomorphisms $T_0\otimes B\simeq B$ and $B\otimes T_0\simeq B$
given by the identity maps $t_0\otimes b\mapsto b$ and $b\otimes t_0\mapsto b$, and 
it gives the same map on $T_0\otimes T_0$. 

\begin{lemma}
The category of $\mathfrak g$-crystals is a monoidal category whose unit object 
is $T_0$. 
\end{lemma}

\begin{remark}
Let $\mathcal C$ be a monoidal category with unit object $I$, 
$B, B'$ two objects of $\mathcal C$. Recall that 
$B'$ is the left dual of $B$ and $B$ is the right dual of $B'$ if 
there exist
$$
\epsilon:I\rightarrow B\otimes B',\;\;
\eta:B'\otimes B\rightarrow I
$$
such that the composition 
$$
B\simeq I\otimes B\overset{\epsilon\otimes\id_B}{\longrightarrow}(B\otimes B')\otimes B
\simeq B\otimes(B'\otimes B)\overset{\id_B\otimes\eta}{\longrightarrow}
B\otimes I\simeq B
$$
is equal to $\id_B$ and the composition 
$$
B'\simeq B'\otimes I\overset{\id_{B'}\otimes\epsilon}{\longrightarrow}
B'\otimes(B\otimes B')\simeq(B'\otimes B)\otimes B'
\overset{\eta\otimes\id_{B'}}{\longrightarrow}I\otimes B'\simeq B'
$$
is equal to $\id_{B'}$. $\mathcal C$ is called \emph{rigid} if every object has 
the left and the right duals. 
The category of $\mathfrak g$-crystals is not a 
rigid monoidal category. To see this, let $B(0)$ be the crystal $\{b_0\}$ with 
$\wt(b_0)=0$, $\epsilon_i(b_0)=\varphi_i(b_0)=0$, $\tilde e_ib_0=\tilde f_ib_0=0$. 
For any $B$, we have that $\varphi_i(b\otimes b_0)\neq-\infty$, which implies 
that there does not exist nonzero crystal morphism 
$B\otimes B(0)\sqcup\{0\}\rightarrow T_0\sqcup\{0\}$ nor 
$T_0\sqcup\{0\}\rightarrow B(0)\otimes B\sqcup\{0\}$. 
Hence, $B(0)$ does not have the dual. 
\end{remark}
\begin{remark}
The category of $\mathfrak g$-crystals is not a braided monoidal category. 
For example, 
$sl_2$-crystals $B(0)\otimes T_{n\omega}$ and $T_{n\omega}\otimes B(0)$ 
are not isomorphic if $n\neq0$. Below we introduce crystals 
which come from integrable $U_v(\mathfrak g)$-modules. For such crystals 
we have isomorphisms $B_1\otimes B_2\simeq B_2\otimes B_1$, 
but we have to choose them functorial and 
they must satisfy the commutativity of moving $B_1$ step by step to the right
$$B_1\otimes B_2\otimes B_3\rightarrow B_2\otimes B_1\otimes B_3\rightarrow 
B_2\otimes B_3\otimes B_1$$ with swapping $B_1$ and $B_2\otimes B_3$ at once. 
For $\mathfrak g=sl_2$ this is not satisfied. 
\end{remark}

\begin{remark}
In the case when $\mathfrak g$ is of affine type, we may consider 
$\mathfrak g'=[\mathfrak g,\mathfrak g]$-crystal, which is obtained by replacing 
$P$ with $P_{cl}$, which is $P$ modulo the null root, in the definition 
of $\mathfrak g$-crystal. Then we have other examples of 
$B_1\otimes B_2\simeq B_2\otimes B_1$ given by combinatorial $R$-matrices 
for the affinizations of finite $\mathfrak g'$-crystals. 
\end{remark}

\begin{definition}
A crystal $B$ is \emph{seminormal} if 
$$
\epsilon_i(b)=\max\{n\in\Z_{\geq0}\mid \tilde e_i^nb\neq0\}\;\;\text{and}\;\;
\varphi_i(b)=\max\{n\in\Z_{\geq0}\mid \tilde f_i^nb\neq0\}
$$
hold, for all $b\in B$.
\end{definition}

\begin{remark}
Let $U_v(\mathfrak l_{ij})$ be the subalgebra of $U_v(\mathfrak g)$ generated by 
$e_i, e_j, f_i, f_j$ and $v^h$, for $h\in P^{\vee}$. 
Let $\Lambda\in P$ be such that
$\langle h_i,\Lambda\rangle\geq0$ and $\langle h_j,\Lambda\rangle\geq0$.
Then, as we will explain below, we have the $\mathfrak l_{ij}$-crystal $B_{ij}(\Lambda)$
which is the crystal of the integrable highest weight $U_v(\mathfrak l_{ij})$-module
with highest weight $\Lambda$.

Let $B$ be a $\mathfrak g$-crystal and
consider it as a $\mathfrak l_{ij}$-crystal. If it is isomorphic to direct sum of 
$B_{ij}(\Lambda)$'s, for all $i, j\in I$ such that $\mathfrak l_{ij}$ is of
finite type, we say that $B$ is \emph{normal}.
The following is proved in \cite[5.2]{Jo}.
\end{remark}

\begin{lemma}
Let $B_1, B_2$ be seminormal (resp. normal) crystals. Then
\begin{enumerate}
\item[(1)]
$B_1\otimes B_2$ is a seminormal (resp. normal) crystal. 
\item[(2)]
Any crystal morphism $f:B_1\rightarrow B_2$ is a strict crystal morphism.
\end{enumerate}
\end{lemma}

\begin{corollary}
The category of seminormal (resp. normal) $\mathfrak g$-crystals 
is a monoidal category whose unit object is $B(0)$. 
\end{corollary}

Recall that $\mathcal K=\mathbb Q(v)$, 
$\A_0=\{c(v)\in\mathcal K\mid \text{$c(v)$ is regular at $v=0$}\}$.  
Assume that the Cartan matrix $A$ is symmetrizable. Then we have 
the $\mathcal K$-algebra $U_v(\mathfrak g)$, 
the quantized enveloping algebra associated with the root datum 
$(A,\Pi,\Pi^{\vee},P,P^{\vee})$. 
Let $\mathcal O_{\rm int}$ be the full category of the BGG category consisting of 
integrable modules. Namely, 
$\mathcal O_{\rm int}$ consists of those 
$M\in U_v(\mathfrak g)\text{-}\Mod$ that satisfies 
\begin{enumerate}
\item[(i)]
$M$ admits a weight decomposition $M=\oplus_{\lambda\in P}M_\lambda$ such that 
$\dim_{\mathcal K}M_\lambda<\infty$. 
\item[(ii)]
There exists a finite set $U\subseteq P$ such that if $M_\lambda\neq0$ then 
$$\lambda\in U-\sum_{i\in I}\Z_{\geq0}\alpha_i.$$
\item[(iii)]
The Chevalley generators $e_i$ and $f_i$ act locally nilpotently on $M$. 
\end{enumerate}

Let $M\in\mathcal O_{\rm int}$. Then we may define 
$\tilde e_i, \tilde f_i:M\rightarrow M$ by 
$\tilde e_if_i^{(n)}u=f_i^{(n-1)}u$ and 
$\tilde f_if_i^{(n)}u=f_i^{(n+1)}u$, for $u\in\Ker e_i$. Here, 
$f_i^{(n)}$ is the $n^{th}$ divided power. 

\begin{definition}
Let $M\in\mathcal O_{\rm int}$. 
An $\A_0$-submodule $\mathcal L=\oplus_{\lambda\in P}\mathcal L_\lambda$ 
is called a \emph{crystal lattice} of $M$ if 
$\mathcal L_\lambda\subseteq M_\lambda$ and 
$\mathcal L_\lambda\otimes\mathcal K=M_\lambda$, for all $\lambda\in P$, 
$\tilde e_i\mathcal L\subseteq\mathcal L$ and 
$\tilde f_i\mathcal L\subseteq\mathcal L$, for all $i\in I$. 
\end{definition}

\begin{definition}
Let $M\in\mathcal O_{\rm int}$. A \emph{crystal basis} of $M$ is a pair 
$(\mathcal L, B=\sqcup_{\lambda\in P} B_\lambda)$ such that 
\begin{enumerate}
\item[(i)]
$\mathcal L=\oplus_{\lambda\in P}\mathcal L_\lambda$ is a crystal lattice of $M$,
\item[(ii)]
$B_\lambda$ is a $\mathbb Q$-basis of $\mathcal L_\lambda/v\mathcal L_\lambda$, 
for all $\lambda\in P$.
\item[(iii)]
$\tilde e_iB\subseteq B\sqcup\{0\}$ and 
$\tilde f_iB\subseteq B\sqcup\{0\}$, for all $i\in I$. 
\item[(iv)]
Let $b, b'\in B$. Then $\tilde f_ib=b'$ if and only if $\tilde e_ib'=b$.
\end{enumerate}
\end{definition}

If $(\mathcal L,B)$ is a crystal basis of $M\in\mathcal O_{\rm int}$, 
then $B$ is a normal $\mathfrak g$-crystal. There are seminormal crystals 
which are not of this form. For normal crystals, no such example is known.
The following theorem was proved by the famous grand loop argument.

\begin{theorem}[Kashiwara]
Let $M\in \mathcal O_{\rm int}$. Then there exists a unique crystal basis up to 
automorphism of $M$. 
\end{theorem}

Let $\Lambda$ be a dominant integral weight. Then the irreducible highest weight
$U_v(\mathfrak g)$-module 
$V_v(\Lambda)$ belongs to $\mathcal O_{\rm int}$. The crystal basis of 
$V_v(\Lambda)$ is denoted by $(\mathcal L(\Lambda),B(\Lambda))$. The highest vector 
$v_\Lambda\in V_v(\Lambda)$ defines the highest weight element $u_\Lambda\in B(\Lambda)$. 

\begin{remark}
As $\mathcal O_{\rm int}$ is a semisimple category, 
every object is a direct sum of $V_v(\Lambda)$'s, which corresponds to 
the direct sum of $B(\Lambda)$'s in the category of crystals. 
$$
\Hom(B(\Lambda),B(\Lambda'))=\begin{cases}
0&(\Lambda\neq\Lambda')\\
\{0,\id_{B(\Lambda)}\}&(\Lambda=\Lambda')\end{cases}
$$
Hence $\Hom_{\mathcal O_{\rm int}}(V_v(\Lambda),V_v(\Lambda'))$ is a 
\lq\lq linearization\rq\rq of 
$\Hom(B(\Lambda),B(\Lambda'))$\footnote{That we adopt the definition of crystal
morphism in \cite{Jo} is important here.}, and $\mathcal O_{\rm int}$ is well controlled 
by the category of crystals. However, it is no more true when we 
compare their monoidal structures. 

Recall that $\mathcal O_{\rm int}$ is a braided monoidal category. 
(It is not rigid in general as long as we adopt the usual definition of the dual for 
general Hopf algebras: 
the dual $V_v(-\Lambda)=\Hom_{\mathcal K}(V_v(\Lambda),\mathcal K)$ is 
the lowest weight module, which does not belong to $\mathcal O_{\rm int}$ 
unless $\mathfrak g$ is of finite type.)  
Hence we have a natural isomorphism
$V_v(\Lambda)\otimes V_v(\Lambda')\simeq V_v(\Lambda')\otimes V_v(\Lambda)$ and 
this implies that $B(\Lambda)\otimes B(\Lambda')$ and 
$B(\Lambda')\otimes B(\Lambda)$ are isomorphic. 
However, as is mentioned above, the subcategory of these crystals is not braided. 
When $\mathfrak g$ is of finite type, Henriques and Kamnitzer \cite{HK} gave 
the notion of commutator for the full category of crystals which consists of 
direct sums of $B(\Lambda)$'s, and showed that 
it gives a coboundary monoidal structure. 
\end{remark}

The crystal $B(\Lambda)$ produces a remarkable basis of $V_v(\Lambda)$.

\begin{theorem}[Kashiwara]
Let $\mathcal L_0=\mathcal L(\Lambda)$ such that $(\mathcal L_0)_\Lambda=\A_0v_\Lambda$. 
Define the bar operation on $V_v(\Lambda)$ by $\overline{v_{\Lambda}}=v_{\Lambda}$ 
and $\overline{f_iu}=f_i\overline u$, for $i\in I$ and $u\in V_v(\Lambda)$, 
and denote the Kostant-Lusztig form of $U_v(\mathfrak g)$ by $U_\A(\mathfrak g)$. Set 
$$
\mathcal L_\infty=\overline{\mathcal L(\Lambda)}\;\;\text{and}\;\;
L=U_\A(\mathfrak g)\mathbb Qv_\Lambda.
$$
Then $(L,\mathcal L_0,\mathcal L_\infty)$ is a balanced triple. In particular, 
we have the canonical basis $\{G(b)\mid b\in B(\Lambda)\}$ of $V_v(\Lambda)$. 
\end{theorem}

Let us consider $U_v^-(\mathfrak g)$. Then it may be viewed as a module over 
the Kashiwara algebra (the algebra of deformed bosons), and we may define 
its crystal basis by the similar recipe.
The crystal so obtained is the crystal $B(\infty)$ and we also have the 
canonical basis $\{G(b)\mid b\in B(\infty)\}$ of
$U_v^-(\mathfrak g)$.\footnote{As is well-known,
Lusztig constructed the basis by geometrizing Ringel's work when
the generalized Cartan matrix is symmetric.}
We have $G(b)v_\Lambda=G(b')$, for a unique $b'\in B(\Lambda)$, or $G(b)v_\Lambda=0$. 
This defines a strict crystal epimorphism 
$B(\infty)\otimes T_\Lambda\rightarrow B(\Lambda)$. 

\begin{theorem}[Kashiwara]
There is an embedding $B(\Lambda)\rightarrow B(\infty)\otimes T_\Lambda$, for 
each dominant integral weight $\Lambda$, such that 
\begin{itemize}
\item[(1)]
The morphisms $B(\Lambda)\otimes T_{-\Lambda}\rightarrow B(\infty)$ 
form an inductive system and 
$$
B(\infty)=\lim_{\Lambda\to\infty} B(\Lambda)\otimes T_{-\Lambda}.
$$
\item[(2)]
The embedding is the section of $B(\infty)\otimes T_\Lambda\rightarrow B(\Lambda)$. 
\end{itemize}
\end{theorem}

We record two theorems which are useful to identify a crystal with $B(\Lambda)$.
The first is by Joseph \cite[6.4.21]{Jo} and the second is by Kashiwara and Saito \cite{KS}.

\begin{theorem}[Joseph]
Suppose that we are given a seminormal crystal $D(\Lambda)$, for each 
dominant integral weight $\Lambda$, such that 
\begin{itemize}
\item[(i)]
There exists an element $d_\Lambda\in D(\Lambda)$ of weight $\Lambda$ and 
all the other elements of $D(\Lambda)$ are of the form 
$\tilde f_{i_1}\cdots\tilde f_{i_N}d_\Lambda$, for some 
$i_1,\dots,i_N\in I$. 
\item[(ii)]
The subcrystal of $D(\Lambda)\otimes D(\Lambda')$ that is generated by 
$d_\Lambda\otimes d_{\Lambda'}$ is isomorphic to $D(\Lambda+\Lambda')$. 
\end{itemize}
Then $D(\Lambda)\simeq B(\Lambda)$, for all $\Lambda$. 
\end{theorem}

\begin{theorem}[Kashiwara-Saito]
Let $B$ be a $\mathfrak g$-crystal, $b_0\in B$ an element of weight $0$. Suppose that 
\begin{enumerate}
\item[(i)]
$\wt(B)\subseteq\sum_{i\in I}\Z_{\leq0}\alpha_i$.
\item[(ii)]
$b_0$ is the unique element of $B$ of weight $0$ and $\epsilon_i(b_0)=0$, for 
all $i\in I$. 
\item[(iii)]
$\epsilon_i(b)$ is finite, for all $i\in I$ and $b\in B$.
\item[(iv)]
There exists a strict embedding 
$\Psi_i:B\rightarrow B\otimes B_i$, for all $i\in I$. 
\item[(v)]
$\Psi_i(B)\subseteq\{b\otimes b_i(a)\mid b\in B, a\in\Z_{\leq0}\}$.
\item[(vi)]
If $b\neq b_0$ then there exists $i\in I$ such that 
$\Psi_i(b)\in\{b\otimes b_i(a)\mid b\in B, a\in\Z_{<0}\}$.
\end{enumerate}
Then $B\simeq B(\infty)$. If there also exists a seminormal crystal $D$, a dominant 
integral weight $\Lambda$ and 
an element $d_\Lambda\in D$ of weight $\Lambda$ such that
\begin{enumerate}
\item[(v)]
$d_\Lambda$ is the unique element of $D$ of weight $\Lambda$.
\item[(vi)]
There is a strict epimorphism 
$\Phi:B\otimes T_\Lambda\rightarrow D$ such that 
$\Phi(b_0\otimes t_\Lambda)=d_\Lambda$. 
\item[(vii)]
$\Phi$ maps $\{b\otimes t_\Lambda\in B\otimes T_\Lambda\mid \Phi(b\otimes t_\Lambda)\neq0\}$ 
to $D$ bijectively. 
\end{enumerate}
Then $D\simeq B(\Lambda)$ and the section of $\Phi$ given 
by the bijective map in {\rm (vii)} gets identified with the embedding 
$B(\Lambda)\rightarrow B(\infty)\otimes T_\Lambda$. 
\end{theorem}

Kashiwara constructed a strict embedding
$\Psi_i:B(\infty)\rightarrow B(\infty)\otimes B_i$ that satisfies 
$\Psi_i(b_0)=b_0\otimes b_i(0)$, for any $i\in I$, and showed that 
such an embedding is unique. 
Hence, $\Psi_i$ in the above theorem is identified with this embedding.

Recall that we have an anti-automorphism of $U_v^-(\mathfrak g)$ defined by 
$f_i^*=f_i$. 
It induces the star crystal structure on $B(\infty)$ defined by
\begin{gather*}
\wt^*(b)=\wt(b^*),\;\;
\epsilon_i^*(b)=\epsilon_i(b^*),\;\;\varphi_i^*(b)=\varphi_i(b^*),\\
\tilde e_i^*b=(\tilde e_ib^*)^*,\;\;\tilde f_i^*b=(\tilde f_ib^*)^*.
\end{gather*}

The next proposition is from \cite[8.1,8.2]{K1}.

\begin{proposition}[Kashiwara]
Let $\Psi_i$ and $\Lambda$ be as above. 
\begin{enumerate}
\item[(1)]
The image of the strict embedding $\Psi_i$ is given by
$$
\{b\otimes b_i(a)\mid b\in B(\infty), \epsilon_i^*(b)=0, a\leq0\}.
$$
\item[(2)]
The image of the embedding $B(\Lambda)\rightarrow B(\infty)\otimes T_\Lambda$ is given by
$$
\{b\otimes t_\Lambda\mid \epsilon_i^*(b)\leq\langle h_i,\Lambda\rangle, 
\text{for any }i\in I\}.
$$
\end{enumerate}
\end{proposition}

\subsection{Realizations of crystals}

Kashiwara crystal has many realizations. Each realization has its own advantage 
and in the case when we may transfer a result in one realization to a result 
in the other realization, it would lead to a very nontrivial consequence. 
This is exactly the case when we apply the theory of crystals to the modular 
representation theory of Hecke algebras. We have obtained 
classification of simple modules, decomposition matrices, representation type of 
the whole algebra, the modular branching rule, so far. This is the aim of the 
next subsection, and as a preparation for this, we explain various realizations here. 

\begin{itemize}
\item[(1)]
Realization by crystal bases: This is already explained. When $\mathfrak g$ is of 
type $A^{(1)}_{e-1}$, it is closely related to soliton theory and 
solvable lattice models. 

Recall that the study of the Kadomtzev-Petviashvili equations by Sato school 
lead to the understanding of the Fock space as a 
$\mathfrak gl_\infty$-module, and then, by reduction to 
the Korteweg-de Vries equation etc, we obtain a $\mathfrak g$-module. 

Let $\lambda=(\lambda_1,\lambda_2,\dots)$ be a partition. 
Then, we assign its beta numbers, in which we have an ambiguity. 
However, this ambiguity precisely amounts to the choice of the coloring 
of the nodes of $\lambda$, or equivalently, the choice of the 
highest weight for the vacuum. Let us fix $m\in\Z/e\Z=I$. 
For a node $x\in\lambda$ which lies in the $a^{th}$ row and the $b^{th}$ column, 
we color $x$ with its residue $r(x)\in\Z/e\Z$ defined by $m-a+b$. 
Then $\mathcal F=\oplus_{\lambda} \mathbb Q\lambda$, the space of these 
colored partitions, becomes a $\mathfrak g$-module 
via
$$
e_i\lambda=\sum_{\mu:r(\lambda/\mu)=i}\mu,\;\;
f_i\lambda=\sum_{\mu:r(\mu/\lambda)=i}\mu
$$
and definitions for the Cartan part. We denote the module by 
$\mathcal F(\Lambda_m)$. The Fock space is 
$\oplus_{m\in\Z/e\Z}\mathcal F(\Lambda_m)$, although we call 
$\mathcal F(\Lambda_m)$'s also Fock spaces. 

Later, they studied the XXZ model with periodic 
boundary conditions. The XXZ model is one of 
the important models for spin chains. 
Then, they found $U_v(sl_2)$-symmetry in the model, and they introduced 
the deformed Fock space $\mathcal F_v=\oplus_{\lambda} \mathcal K\lambda$, 
which becomes a $U_v(\mathfrak g)$-module after a choice of the coloring 
of partitions. We denote the module by 
$\mathcal F_v(\Lambda_m)$. The vacuum has the weight $\Lambda_m$, 
and the $U_v(\mathfrak g)$-submodule generated by the vacuum is 
isomorphic to $V_v(\Lambda_m)$. The observation was that the space of states of 
half infinite spin chains looks like $\oplus_{m\in\Z/e\Z}\mathcal F_v(\Lambda_m)$. 

Misra and Miwa showed that 
$\oplus_{\lambda}\A_0\lambda$ is a crystal lattice of $\mathcal F_v(\Lambda_m)$, 
and that the set of partitions is its crystal. Thus, its connected component 
that contains the empty partition is isomorphic to $B(\Lambda_m)$. 
Note that $B(\Lambda_m)$ is realized as a subcrystal of the crystal of partitions. 

\item[(2)]
Realization by Young diagram/Young tableaux and Young walls: This gave new treatment 
of classical objects in algebraic combinatorics. For example, the set of 
$e$-restricted/$e$-regular colored partitions is a realization of $B(\Lambda_m)$. 
This is the crystal which appeared in (1) as the connected component which 
contains the empty partition. 

\item[(3)]
Path realization: This came from the same effort to understand the spin model. 
Let $\mathfrak g$ be of affine type. Following standard notation, 
we have a special node $0\in I$ and the Cartan subalgebra 
has the basis $\{h_i\mid i\in I\}\sqcup\{d\}$. The canonical central element 
$c=\sum_{i\in I}c_ih_i$ is defined by the requirement that 
${\rm gcd}\{c_i\in\Z_{>0}\mid i\in I\}=1$. 
Let $\Lambda$ be a dominant integral weight such that $\langle d,\Lambda\rangle=0$. 
Then the ground state 
$p_{\Lambda}=\cdots\otimes b_1\otimes b_0$, which corresponds to the highest weight 
element, and the other excited states 
$\cdots\otimes p_1\otimes p_0$ where $p_k=b_k$, for sufficiently large $k$, 
form the crystal $B(\Lambda)$, which explained the appearance of the crystal 
in the XXZ model. In the XXZ model, we have $b_k=+$ or $b_k=-$, for all $k$. 
In the crystal language, $\{+,-\}$ is a perfect crystal. 
\begin{definition}
Let $l$ be a positive integer, which is called a level. 
Let $B$ be a finite $\mathfrak g'$-crystal. 
We say that $B$ is a \emph{perfect crystal of level $l$} if 
it satisfies the following. 
\begin{enumerate}
\item[(i)]
There exists a finite dimensional $U_v(\mathfrak g')$-module with crystal basis 
$(\mathcal L,B)$, for some crystal lattice $\mathcal L$. 
\item[(ii)]
$B\otimes B$ is connected.
\item[(iii)]
There exists $\lambda_0\in P_{cl}$ such that
$\wt(B)\subseteq\lambda_0-\sum_{i\neq0}\Z_{\geq0}\alpha_i$.
\item[(iv)]
There is the unique element of weight $\lambda_0$ in $B$. 
\item[(v)]
$\langle c,\epsilon(b)\rangle:=\sum c_i\epsilon_i(b)\geq l$, for all $b\in B$. 
\item[(vi)]
For any dominant integral weight 
$\lambda=\sum_{i\in I}\lambda_i\Lambda_i\in P_{cl}=\oplus_{i\in I}\Z\Lambda_i$ 
with $\langle c,\lambda\rangle=l$, there exist unique vectors $b^\lambda\in B$ and 
$b_\lambda\in B$ such that $\epsilon_i(b^\lambda)=\lambda_i$ and 
$\varphi_i(b_\lambda)=\lambda_i$, for all $i\in I$. 
\end{enumerate}
\end{definition}
(v) implies that $B(\lambda)\otimes B$ has the unique highest weight element of weight $\lambda$, 
which is $u_\lambda\otimes b^\lambda$ by (vi). 
Further, every element of $B(\lambda)\otimes B$ 
is of the form $\tilde f_{i_1}\cdots\tilde f_{i_N}(u_\lambda\otimes b^\lambda)$. 

Let $\Lambda\in P$ be dominant with $\langle c,\Lambda\rangle=l$ and 
$\langle d,\Lambda\rangle=0$. Set $\lambda=\overline\Lambda\in P_{cl}$. 
Then Kang, Kashiwara, Misra, Miwa, Nakashima and Nakayashiki showed that 
there exists the unique crystal isomorphism 
$B(\lambda)\simeq B(\lambda')\otimes B$ defined by 
$u_\lambda\mapsto u_{\lambda'}\otimes b_\lambda$, where 
$\lambda'=\sum_{i\in I}\epsilon_i(b_\lambda)\Lambda_i$, and that 
its affinization gives an embedding 
$B(\Lambda)\rightarrow B(\Lambda')\otimes B^{\rm aff}$ defined by
$u_\Lambda\mapsto u_{\Lambda'}\otimes b_\lambda$. 
By iterating the procedure, we obtain the embedding 
$B(\Lambda)\rightarrow (B^{\rm aff})^\infty$. Here, $B^{\rm aff}=
\{z^nb\mid b\in B, n\in\Z\}$ is the seminormal $\mathfrak g$-crystal 
which is defined by setting $\wt(z)=\delta$ and defining $\tilde e_i$ and 
$\tilde f_i$ by $z^{\delta_{i0}}\tilde e_i$ and 
$z^{-\delta_{i0}}\tilde f_i$ respectively. 
This embedding is called the path realization of $B(\Lambda)$. 

\item[(4)]
Littelmann's path model: Littelmann considered piecewise linear paths in 
$P\otimes\mathbb R$. Let $W$ be the Weyl group of $\mathfrak g$. Let 
$\Lambda$ be a dominant integral weight. Given rational numbers 
$0=a_0<\cdots<a_s=1$ and weights $\nu_1,\dots,\nu_s\in W\Lambda$, let 
$\pi(t):[0,1]\rightarrow P\otimes\mathbb R$ be 
the piecewise linear path which goes to the $\nu_j$ direction 
during $t\in [a_{j-1},a_j]$, for $1\leq j\leq s$. We denote the path 
$(\nu_1,\dots,\nu_s;a_0,\dots,a_s)$. A Lakshmibai-Seshadri path is 
such a path which satisfies certain additional condition.
Joseph and Kashiwara showed that 
the set of Lakshmibai-Seshadri paths is a crystal which is isomorphic to 
$B(\Lambda)$. Kashiwara introduced different treatment of the path model, 
which is more useful for us. 
Let $B$ and $B'$ be crystals. A map $\psi:B\rightarrow B'$ is
called a \emph{crystal morphism of amplitude $h$} if
\begin{gather*}
\wt(\psi(b))=h\wt(b),\;\;\epsilon_i(\psi(b))=h\epsilon_i(b),\;\;
\varphi_i(\psi(b))=h\varphi_i(b),\\
\psi(\tilde e_ib)=\tilde e_i^h\psi(b),\;\;
\psi(\tilde f_ib)=\tilde f_i^h\psi(b), 
\end{gather*}
for all $b\in B$. Kashiwara showed that 
there is a unique crystal isomorphism of amplitude $h$ for  
$B(\Lambda)\rightarrow B(h\Lambda)\subseteq B(\Lambda)^{\otimes h}$, 
and that if $h$ is sufficiently divisible then the image stabilizes in the 
sense that there exist $b_1,\dots,b_s$ and $0=a_0<\cdots<a_s=1$ which are 
independent of $h$ such that 
$$
b\mapsto b_1^{\otimes ha_1}\otimes b_2^{\otimes h(a_2-a_1)}\otimes\cdots
\otimes b_s^{\otimes h(1-a_{s-1})},
$$
for sufficiently divisible $h$. Set $\nu_j=\wt(b_j)$, for $1\leq j\leq s$. 
Then the map $b\mapsto (\nu_1,\dots,\nu_s;a_0,\dots,a_s)$ gives 
the Littelmann path model realization. When $\mathfrak g$ is of type 
$A^{(1)}_{e-1}$ and $B(\Lambda)=B(\Lambda_m)$, which is realized on partitions, 
then $b_j$ are $e$-cores. Thus, a variant of the Littelmann path model is 
that $b\in B(\Lambda_m)$ is expressed by $(\nu_1,\dots,\nu_s;a_0,\dots,a_s)$, 
where $\nu_1\supseteq\cdots\supseteq\nu_s$ are $e$-cores. 

\item[(5)]
Polyhedral realization: This is the embedding $B(\infty)\rightarrow 
\cdots\otimes B_{i_2}\otimes B_{i_1}$ given by the 
strict morphisms $\Psi_i:B(\infty)\rightarrow B(\infty)\otimes B_i$, 
for $i\in I$. This induces the embedding
$$
B(\Lambda)\rightarrow \cdots\otimes B_{i_2}\otimes B_{i_1}\otimes T_\Lambda.
$$
\end{itemize}

There are other realizations in terms of irreducible components of Nakajima 
quiver varieties, Mirkovic-Vilonen cycles/polytopes, Nakajima monomials, etc. 

\subsection{Kashiwara crystals and Hecke algebras}

Let us return to our theme, Hecke algebras. We mainly focus on type $B$ or
its generalization to type $(d,1,n)$. 
Brou\'e, Malle and Rouquier introduced cyclotomic Hecke algebras. 
Let $W$ be a complex reflection group, $\mathcal A$ the arrangement 
consisting of reflection hyperplanes in the defining $W$-module $V$. 
Then they introduced the 
Knizhnik-Zamolodchikov equation on the complement 
$V\setminus\mathcal A$, 
and the cyclotomic Hecke algebra is defined as a quotient of the group algebra of
$\pi_1((V\setminus\mathcal A)/W)$. When $W=G(d,1,n)$, we have the AK-algebra over $\C$, 
which is obtained from the definition of $\H(W(B_n),S,L)$ by replacing 
the quadratic relation for $T_0$ with 
$(T_0-q^{\gamma_1})\cdots (T_0-q^{\gamma_d})=0$. Hecke algebras of type B 
are special cases of these algebras. Assume that $q$ is a primitive $e^{th}$ 
root of unity and $e\geq2$. Then we consider $\mathfrak g$ of type $A^{(1)}_{e-1}$ 
and the dominant integral weight $\Lambda=\sum_{j=1}^d\Lambda_{\gamma_j}$. 
We have $V_v(\Lambda_{\gamma_j})\subseteq\mathcal F_v(\Lambda_{\gamma_j})$, 
for $1\leq j\leq d$. By using the coproduct 
\begin{gather*}
\Delta(e_i)=1\otimes e_i+e_i\otimes v^{-h_i},\;\;
\Delta(f_i)=f_i\otimes 1+v^{h_i}\otimes f_i\\
\Delta(v^h)=v^h\otimes v^h,\;
\text{for $h\in P^{\vee}$},
\end{gather*}
we define 
$$
\mathcal F_v(\Lambda)=\mathcal F_v(\Lambda_{\gamma_d})\otimes\cdots\otimes
\mathcal F_v(\Lambda_{\gamma_1}).
$$
Denote $\lambda^{(d)}\otimes\cdots\otimes\lambda^{(1)}$ by 
$\underline\lambda=(\lambda^{(1)},\dots,\lambda^{(d)})$. The set of 
colored multipartitions becomes the $\mathfrak g$-crystal 
associated with $\mathcal F_v(\Lambda)$. The $U_v(\mathfrak g)$-submodule 
generated by the vacuum of $\mathcal F_v(\Lambda)$ is isomorphic to 
$V_v(\Lambda)$ and it defines the connected component of the crystal of 
the colored multipartitions. We say that $\underline\lambda$ is 
\emph{Kleshchev}  
if it belongs to the component. Hence, we have a realization of $B(\Lambda)$ 
by Kleshchev multipartitions, and we identify them hereafter. 
Now, we have the canonical basis $\{G(\underline\mu)\mid \underline\mu:\text{Kleshchev}.\}$ 
of $V_v(\Lambda)\subseteq\mathcal F_v(\Lambda)$. We may expand $G(\underline\mu)$ in 
$\mathcal F_v(\Lambda)$ and write 
$$
G(\underline\mu)=\sum_{\underline\lambda} 
d_{\underline\lambda,\underline\mu}(v)\underline\lambda.
$$
It is known that $d_{\underline\lambda,\underline\mu}(v)\in v\Z_{\geq0}[v]$ if 
$\underline\lambda\neq\underline\mu$, and 
$d_{\underline\lambda,\underline\lambda}(v)=1$. We specialize at $v=1$ and write
$G(\underline\mu)=\sum_{\underline\lambda} 
d_{\underline\lambda,\underline\mu}(1)\underline\lambda$ by abuse of notation. 
This is an equality in the nondeformed Fock space $\mathcal F(\Lambda)$. 
We denote the $\mathfrak g$-submodule of $\mathcal F(\Lambda)$ generated by 
the vacuum by $V(\Lambda)$. 

Denote by $\H_n^\Lambda(q)$ the cyclotomic Hecke algebra whose parameters are 
specified above. Dipper, James and Mathas showed that $\H_n^\Lambda(q)$ is 
a cellular $R$-algebra and the poset is the set of multipartitions. 
The cell modules are called \emph{Specht modules} and we denote them by 
$S^{\underline\lambda}$. $D^{\underline\lambda}$ is the module obtained 
by factoring out the radical of the bilinear form on $S^{\underline\lambda}$. 
An old theorem of mine then says the following. 
In (1), we see that the cellular algebra structure fits well in the 
crystal picture. 

\begin{theorem}[A]
Let $\mathfrak g$, $\Lambda$, $V(\Lambda)\subseteq\mathcal F(\Lambda)$ and 
$\H_n^\Lambda(q)$ be as above.
\begin{enumerate}
\item[(1)]
$D^{\underline\mu}\neq0$ if and only if $\underline\mu$ is Kleshchev. Hence, 
the union for $n\ge 0$ of the set of (isomorphism classes of) simple
$\H_n^\Lambda(q)$-modules has a structure
of $\mathfrak g$-crystal, which is isomorphic to $B(\Lambda)$.
\item[(2)]
Let $K_n(\Lambda)=K(\H_n^\Lambda(q)\text{\rm -mod})\otimes\mathbb Q$, or 
equivalently, $K^{\rm split}(\H_n^\Lambda(q)\text{\rm -proj})\otimes\mathbb Q$. Then 
$K(\Lambda)=\oplus_{n\geq0}K_n(\Lambda)$ becomes a $\mathfrak g$-module, which is 
isomorphic to the $\mathfrak g$-module $V(\Lambda)$. 
\item[(3)]
Identify $K(\Lambda)$ with $V(\Lambda)\subseteq\mathcal F(\Lambda)$ 
in the unique way by the requirement that 
$D^{\underline\emptyset}=P^{\underline\emptyset}$ corresponds to 
$\underline\emptyset$. 
Let $P^{\underline\mu}$ be the projective cover of 
$D^{\underline\mu}$. Then we have 
$$
[P^{\underline\mu}]=\sum_{\underline\lambda}
d_{\underline\lambda,\underline\mu}\underline\lambda,
$$
where $d_{\underline\lambda,\underline\mu}=[S^{\underline\lambda}:D^{\underline\mu}]$ 
are decomposition numbers. 
\item[(4)]
Suppose that the characteristic of $R$ is $0$. Then 
$[P^{\underline\mu}]=G(\underline\mu)$. In particular, we have 
$d_{\underline\lambda,\underline\mu}=d_{\underline\lambda,\underline\mu}(1)$. 
\end{enumerate}
\end{theorem}

The proof uses results of Lusztig and Ginzburg on affine Hecke algebras and 
Lusztig's construction of $U_v^-(\mathfrak g)$, which is the generic composition 
algebra of the Ringel-Hall algebra of the cyclic quiver. In \cite{A2} and \cite{A4}
we explained the materials which were used in the proof of the theorem.
Note that the Hall polynomials of the cyclic quiver were given by Jin Yun Guo.
A generalization of
the above theorem to affine Hecke algebras of type $B$ is attempted by Enomoto and
Kashiwara. This involves a new type of crystals, called symmetric crystals. 

The theorem also suggests how to label block algebras. Recall that the block 
algebras of Hecke algebras of type $A$ are labelled by $e$-cores. 
Recall also that if $\mathfrak g$ is of twisted type or 
simply-laced nontwisted type and 
$\Lambda$ has level $1$, then the weight poset of $V(\Lambda)$ 
has the form $\sqcup_{\nu\in W\Lambda}(\nu-\Z_{\geq0}\delta)$. 
Since the $W$-orbit through the empty partition is the set of $e$-cores, 
the set of partitions $\mu$ with $\wt(\mu)\in \nu-\Z_{\geq0}\delta$ 
has the unique $e$-core of weight $\nu$. 
Thus, two partitions $\lambda$ and $\lambda'$ have the same $e$-core 
if and only if $\wt(\lambda)$ and $\wt(\lambda')$ belong to the same 
$\nu-\Z_{\geq0}\delta$. However, 
${\rm ht}(\Lambda-\lambda)={\rm ht}(\Lambda-\lambda')=n$ implies that 
this is equivalent to $\wt(\lambda)=\wt(\lambda')$. In conclusion, 
the block algebras of $\H_n^{\Lambda_m}(q)$ are parametrized by 
the weight set of $V(\Lambda_m)$. This picture is proved to be true 
for general $\Lambda$ by Lyle and Mathas \cite{LM}.

\begin{theorem}[Lyle-Mathas]
The block decomposition of $\H_n^\Lambda(q)$ is the same as the weight space 
decomposition $\oplus_{\mu:{\rm ht}(\Lambda-\mu)=n}V(\Lambda)_\mu$. 
In particular, the block algebras are labelled by 
$$\{\mu\in P\mid {\rm ht}(\Lambda-\mu)=n,\;V(\Lambda)_\mu\neq0\}.$$ 
\end{theorem}

Denote by $B^\Lambda_\mu(q)$ the block algebra labelled by $\mu$. Then 
the set of simple $B_\mu^\Lambda(q)$-modules is 
$\{D^{\underline\lambda}\mid \underline\lambda\in B(\Lambda)_\mu\}$ and 
if the characteristic of $R$ is $0$, we can compute the decomposition 
matrix of $B_\mu^\Lambda(q)$. 

We explain four more applications of the theory of crystals and 
the canonical bases. The first is the modular 
branching rule. Let $D^{\underline\lambda}$ be a simple 
$\H_n^\Lambda(q)$-module. Its restriction to $\H_{n-1}^\Lambda(q)$ is 
not semisimple in general. The modular branching rule 
is an explicit formula to describe which simple 
$\H_{n-1}^{\Lambda}(q)$-modules appear in the socle. 
This is a suitable generalization of the ordinary branching 
problem. In \cite{GV}
Grojnowski and Vazirani proved that 
$\Soc(D^{\underline\lambda}|_{\H_{n-1}^{\Lambda}(q)})$ is multiplicity free. 
They also showed that the set of simple $\H_n^\Lambda(q)$-modules becomes 
a crystal, which is isomorphic to $B(\Lambda)$ again. Note however that 
crystal isomorphisms do not respect the labelling of simple 
$\H_n^\Lambda(q)$-modules. In the next theorem from \cite{A5}, we proved that 
the isomorphism in question is the identity map. 

\begin{theorem}[A]
The modular branching rule is given by
$$\Soc D^{\underline\mu}|_{\H_{n-1}^{\Lambda}(q)}\simeq
\bigoplus_{i\in I=\Z/e\Z} D^{\tilde e_i\underline\mu}.$$ 
\end{theorem}

Thus, the modular branching rule has a very crystal theoretic description. 

The second is about representation type \cite{A3}. 
For Hecke algebras of type $A$, it was settled by Erdmann and Nakano 
by different methods. We have obtained the result for any $L$, but instead of 
preparing further notations, we state the result only in the case when 
$L$ is the length function. Blockwise determination is in progress. 

\begin{theorem}[A]
Let $W$ be a finite Weyl group without exceptional components, 
$P_W(x)=\sum_{w\in W}x^{\ell(w)}$ the Poincar\'e polynomial of $W$. 
We suppose that $R$ is an algebraically closed field. 
Then $\H_R=\H(W,S,L=\ell)\otimes R$ is
\begin{enumerate}
\item[(i)]
semisimple if $P_W(q)\neq0$.
\item[(ii)]
of finite type but not semisimple if 
$$
\max\{k\in\Z_{\geq0}\mid \text{$(x-q)^k$ divides $P_W(x)$.}\}=1.
$$ 
\item[(iii)]
of tame type but not of finite type if $q=-1\neq 1$ and 
$$
\max\{k\in\Z_{\geq0}\mid \text{$(x-q)^k$ divides $P_W(x)$.}\}=2. 
$$
\item[(iii)]
of wild type otherwise. 
\end{enumerate}
\end{theorem}

The third is about an old conjecture of Dipper, James and Murphy. 
When they started the study of $\H_n=\H(W(B_n),S,L)\otimes R$ motivated 
by classification of simple modules of finite groups of Lie type in the 
non-defining characteristic case, Kashiwara crystal was not available.
The Specht module theory they constructed
is the special case of the cellular structure we explained above, 
and they conjectured  
when $D^{\underline\lambda}$ was nonzero. The idea resembles the 
highest weight theory. 
Define the Jucys-Murphy elements $t_1,\dots,t_n$ by 
$t_1=T_0$ and $t_{i+1}=T_it_iT_i$, for $1\leq i\leq n-1$. 
They generate a commutative subalgebra $\mathcal T_n$, 
which plays the role of the Cartan subalgebra. 
\begin{enumerate}
\item[(i)]
One dimensional representations of $\mathcal T_n$ are called \emph{weights}.
\item[(ii)]
For an $\H_n$-module, 
the generalized simultaneous eigenspace decomposition of the module 
is called the \emph{weight decomposition}.
\end{enumerate}
Weights are labelled by bitableaux. 
$\underline\lambda$ is \emph{$(Q,e)$-restricted} if there exists a weight 
which corresponds to a bitableau of shape $\underline\lambda$ 
such that it appears in $S^{\underline\lambda}$ but does not appear in 
$S^{\underline\mu}$, for all $\underline\mu\triangleleft\underline\lambda$. 
The following statement was the conjecture made by 
Dipper, James and Murphy. Jacon and I proved this conjecture in \cite{AJ}
by using one of the main results of \cite{AKT}.

\begin{theorem}[A-Jacon]
$D^{\underline\lambda}\neq0$ if and only if 
$\underline\lambda$ is $(Q,e)$-restricted. 
\end{theorem}

As we already know that $D^{\underline\lambda}\neq0$ if and only if 
$\underline\lambda$ is Kleshchev, what we actually proved is the 
assertion that $\underline\lambda$ is Kleshchev if and only if 
$\underline\lambda$ is $(Q,e)$-restricted.

The fourth is a remark on the Mullineux map for the symmetric
group and the Hecke algebra of type $A$. The algebras have the
involution $T_s\mapsto -T_s^{-1}$ and the involution induces
a permutation of simple modules.
The permutation is described by the transpose of partitions
when the algebra is semisimple, but it was considered to be
difficult to describe the permutation when the algebra is not
semisimple. The Mullineux map was its conjectural description and
it took long time before Kleshchev proved the
conjecture. In \cite{AKT}, we have also proved that
the Mullineux map is always given by transpose of
partitions, if we work in the path model.

\section{Quasihereditary covers of Hecke algebras}
\subsection{The rational Cherednik algebra}
As our main object of study is $\H_n^\Lambda(q)$, 
we focus on the rational Cherednik algebra associated with 
$G(d,1,n)$. To introduce the algebra, we need many notations. Let 
$$
V=\C^n=\C e_1\oplus\cdots\oplus\C e_n,
$$
where $e_1,\dots,e_n$ are standard basis vectors, that is, 
the $i^{th}$ entry of $e_k$ is $\delta_{ki}$. We write elements 
of $V$ by $y=\sum_{i=1}^n y_ie_i$, where $y_i\in\C$, for $1\leq i\leq n$. 
Similarly, let 
$$
V^*=\Hom_\C(V,\C)=\C x_1\oplus\cdots\oplus\C x_n,
$$
where $\langle x_i,e_j\rangle=\delta_{ij}$. We write elements 
of $V^*$ by $\xi=\sum_{i=1}^n \xi_ix_i$, where $\xi_i\in\C$, for $1\leq i\leq n$.
In the rest of the paper, we denote $G(d,1,n)$ by $W$.

\begin{definition}
Let $\zeta=\exp(\frac{2\pi\sqrt{-1}}{d})$. We define $t_i\in W$, for $1\leq i\leq n$, by
$$
t_ie_k=\begin{cases} e_k\quad&(k\neq i)\\
                    \zeta e_k\quad&(k=i),\end{cases}
$$
and $s_{ij;\alpha}\in W$, for $1\leq i<j\leq n$ and $0\leq\alpha<d$, by
$$
s_{ij;\alpha}e_k=\begin{cases} e_k\quad&(k\neq i,j)\\
                    \zeta^{-\alpha}e_j\quad &(k=i)\\
                    \zeta^{\alpha}e_i\quad&(k=j).\end{cases}
$$
\end{definition}
Then the set of complex reflections in $W$ is
$$
S=\{t_i^a \mid 1\leq i\leq n,\; 1\leq a<d\}\sqcup
\{s_{ij;\alpha}\mid 1\leq i<j\leq n,\;0\leq\alpha<d\}.
$$

Let $S_n$ be the symmetric group of degree $n$. $S_n$ acts on 
$V$ by $we_k=e_{w(k)}$, and we may identify 
$S_n$ with $G(1,1,n)\subseteq G(d,1,n)=W$. 
Let $T=\langle t_1,\dots,t_n\rangle$ be the subgroup of $W$ generated by 
$t_1,\dots,t_n$, which is isomorphic to $(\mathbb Z/d\mathbb Z)^n$. 
$W$ is the semidirect product of $T$ (which is a normal subgroup of $W$) and $S_n$. 

We denote the action of $w\in W$ on $V$ by $w:y\mapsto w(y)$, and 
the action of $w\in W$ on $V^*$ by $w:\xi\mapsto w(\xi)$. Note that 
\begin{itemize}
\item[(i)]
if $t=t_1^{a_1}\cdots t_n^{a_n}$ then $t(x_k)=\zeta^{-a_k}x_k$, 
\item[(ii)]
if $w\in S_n$ then $w(x_k)=x_{w(k)}$.
\end{itemize}

For each $s\in S$, we have the reflection hyperplane
$H_s=\{ y\in V\mid s(y)=y\}$. Let $H_i=\Ker(x_i)$ and 
$H_{ij;\alpha}=\Ker(x_i-\zeta^\alpha x_j)$. 
If $s=t_i$ then $H_s=H_i$, and if $s=s_{ij;\alpha}$ then 
$H_s=H_{ij;\alpha}$. We denote the hyperplane arrangement $\{ H_s\mid s\in S\}$ 
by $\mathcal A$.  

\begin{definition}
\begin{itemize}
\item[(1)]
We define the set of roots $\Phi=\{ \alpha_H\mid H\in\mathcal A\}\subseteq V^*$ 
by $\alpha_{H_i}=x_i$ and $\alpha_{H_{ij;\alpha}}=x_i-\zeta^\alpha x_j$. 
\item[(2)]
We define the set of coroots $\Phi^{\vee}=\{ v_H\mid H\in\mathcal A\}\subseteq V$ 
by $v_{H_i}=e_i$ and $v_{H_{ij;\alpha}}=e_i-\zeta^{-\alpha}e_j$.
\end{itemize}
\end{definition}

For each $H\in\mathcal A$, define 
$W_H=\{w\in W\mid w(y)=y, \text{for all $y\in H$.}\}$ and 
$e_H=|W_H|$. If $H=H_i$ then $W_H=\langle t_i\rangle$ 
and $e_H=d$. If $H=H_{ij;\alpha}$ then 
$W_H=\langle s_{ij;\alpha}\rangle$ and $e_H=2$. 

As $V=H\oplus \C v_H$ is a decomposition into a direct sum of 
$W_H$-modules, and $W_H$ acts trivially on $H$, $\C v_H$ affords  
a faithful representation of $W_H$. Thus 
$$
W_H\simeq \{\exp(2\pi\sqrt{-1}a/e_H)\mid 0\leq a<e_H\}\subseteq GL(\C v_H),
$$
and we may define a generator of the cyclic group $w_H\in W_H$ by 
$w_H\leftrightarrow \exp(2\pi\sqrt{-1}/e_H)$. In other words, we define 
$w_H=t_i$ if $H=H_i$ and $w_H=s_{ij;\alpha}$ if $H=H_{ij;\alpha}$. 

\begin{definition}
$\epsilon_{H,k}=\frac{1}{e_H}\sum_{a=0}^{e_H-1}
\exp(2\pi\sqrt{-1}ak/e_H)w_H^a$, for $0\leq k<e_H$. 
\end{definition}

The $W_H$-module $\C\epsilon_{H,k}$ affords the representation 
$w_H\mapsto \exp(-2\pi\sqrt{-1}k/e_H)$. 

\begin{definition}
Let $R$ be a commutative $\C$-algebra. We suppose that parameters 
$\kappa_1,\dots,\kappa_{d-1}, h\in R$ are given. Define 
$$
\underline\kappa=(\kappa_i)_{i\in{\mathbb Z}/d{\mathbb Z}}\;\;
\text{and}\;\;
\underline{h}=(h_i)_{i\in{\mathbb Z}/2{\mathbb Z}}
$$
by extending the kappa parameters by $\kappa_0=\kappa_d=0$ and 
by $h_1=h$, $h_0=h_2=0$. 
Then, for $0\leq k<e_H$, we define 
$$
c_{H,k}=\begin{cases} \kappa_k\;\;&(H=H_i)\\
h_k\;\;&(H=H_{ij;\alpha}).\end{cases}
$$
\end{definition}

\begin{definition}
For $H\in\mathcal A$, we define
\begin{itemize}
\item[(1)]
$\gamma_H=e_H\sum_{k=0}^{e_H-1}(c_{H,k+1}-c_{H,k})\epsilon_{H,k}\in RW_H$. 
\item[(2)]
$a_H=e_H\sum_{k=1}^{e_H-1}c_{H,k}\epsilon_{H,k}\in RW_H$.
\end{itemize}
\end{definition}
If $H=H_{ij;\alpha}$ then $\gamma_H=2hs_{ij;\alpha}$ and 
$a_H=h(1-s_{ij;\alpha})$. 

Now, we are ready to introduce the rational Cherednik algebra. 
Let $R$ be a commutative $\C$-algebra such that parameters 
$\kappa_1,\dots,\kappa_{d-1}, h\in R$ are given. 
The $W$-action on $V$ and $V^*$ naturally defines $W$-action 
on $T(V\oplus V^*)$, and we have the smash product $T(V\oplus V^*)\sharp W$. 
We have the relations $wyw^{-1}=w(y)$ and $w\xi w^{-1}=w(\xi)$, for 
$y\in V$, $\xi\in V^*$ and $w\in W$.

\begin{definition}
The \emph{rational Cherednik algebra} $H_R(\underline\kappa,h)$
(associated with $W$) is the $R$-algebra obtained from the
$R$-algebra $T(V\oplus V^*)\sharp W\otimes_\C R$
by factoring out the two-sided ideal generated by
$$
[y,\xi]-\langle \xi,y\rangle-\sum_{H\in\mathcal A}
\frac{\langle \xi,v_H\rangle\langle\alpha_H,y\rangle}{\langle\alpha_H,v_H\rangle}
\gamma_H,\;\;[y,y'], \;[\xi,\xi'],
$$
where $y,y'$ run through $V$ and $\xi,\xi'$ run through $V^*$. 
\end{definition}

Let $\mathcal D$ be the sheaf of algebraic differential operators on 
$V_{reg}$. Let us consider the trivial $W$-equivariant bundle 
$V_{reg}\times \C W$ on $V_{reg}$ and let $\mathcal O(V_{reg},\C W)$ 
be the global sections. 
$W$ acts on the space by $(w\cdot f)(y)=wf(w^{-1}(y))$ as usual. 
$\mathcal D(V_{reg})$ also acts on the space and we have 
$w\partial_yw^{-1}=\partial_{w(y)}$ and $w\xi w^{-1}=w(\xi)$, where 
$\xi$ is considered as a multiplication operator, and 
$\partial_y=\sum_{i=1}^n y_i\frac{\partial}{\partial x_i}$, for 
$y=\sum_{i=1}^n y_ie_i$. In particular, 
$\mathcal O(V_{reg},\C W)$ is a $\mathcal D(V_{reg})\sharp W$-module. 

\begin{definition}
The operator 
$$
T_y=\partial_y+\sum_{H\in\mathcal A}\frac{\langle\alpha_H,y\rangle}{\alpha_H}a_H
\in \mathcal D(V_{reg})\sharp W\otimes_\C R,
$$
which acts on $\mathcal O(V_{reg},\C W)\otimes_\C R$, is 
called the \emph{Dunkl operator}.
\end{definition}

The following lemma from \cite{EG} is not difficult to prove.

\begin{lemma}[Etingof-Ginzburg]
\item[(1)]
We have an $R$-algebra monomorphism 
$$
H_R(\underline\kappa,h)\longrightarrow \mathcal D(V_{reg})\sharp W\otimes_\C R
$$
given by $\xi\mapsto\xi$, $y\mapsto T_y$ and $w\mapsto w$.
\item[(2)]
$H_R(\underline\kappa,h)=S(V^*)\otimes_\C S(V)\otimes_\C RW$ as 
an $R$-module. In particular, $H_R(\underline\kappa,h)$ is $R$-free of 
infinite rank with basis given by 
$$
\{x_1^{\alpha_1}\cdots x_n^{\alpha_n}e_1^{\beta_1}\cdots e_n^{\beta_n}w\mid
\alpha_i\geq0, \beta_i\geq0, w\in W\}.
$$
\item[(3)]
If we localize the monomorphism in (1), we have 
$$
\mathcal O(V_{reg})\otimes_{\mathcal O(V)}H_R(\underline\kappa,h)
\simeq \mathcal D(V_{reg})\sharp W\otimes_\C R.
$$
In particular, 
$\mathcal O(V_{reg})\otimes_{\mathcal O(V)}H_R(\underline\kappa,h)$ 
is an $R$-algebra. 
\end{lemma}

\begin{definition}
Let $z=\sum_{H\in\mathcal A}a_H\in RW$ and define 
$$
\Eu=-z+\sum_{i=1}^n x_ie_i\in H_R(\underline\kappa,h).
$$
The element $\Eu$ is called the \emph{Euler element}. 
\end{definition}

We have $[\Eu, y]=-y$, $[\Eu, \xi]=\xi$ and $[\Eu, w]=0$. Hence 
$H_R(\underline\kappa,h)$ is a $\Z$-graded $R$-algebra. 

\begin{definition}
Let $\mathcal O$ be the full subcategory of $H_R(\underline\kappa,h)\text{-}\Mod$ 
consisting of the objects that satisfy
\begin{itemize}
\item[(i)]
finitely generated as an $H_R(\underline\kappa,h)$-module, i.e. 
$\mathcal O\subseteq H_R(\underline\kappa,h)\text{-}\fgMod$,
\item[(ii)]
locally nilpotent as a $S(V)$-module.
\end{itemize}
\end{definition}

Recall that an $R$-algebra $A$ is \emph{filtered} if there exists a family 
of $R$-submodules $\{F_p(A)\}_{p\in\mathbb Z_{\geq0}}$ such that
\begin{itemize}
\item[(1)]
$1\in F_0(A)$.
\item[(2)]
$F_p(A)\subseteq F_{p+1}(A)$ and $\cup_{p\geq0} F_p(A)=A$. 
\item[(3)]
$F_p(A)F_q(A)\subseteq F_{p+q}(A)$.
\end{itemize}
An $A$-module $M$ is \emph{filtered} if there exists a family 
of $R$-submodules $\{F_p(M)\}_{p\in\mathbb Z}$ such that
\begin{itemize}
\item[(1)]
$F_p(M)=0$, for sufficiently small $p$.
\item[(2)]
$F_p(M)\subseteq F_{p+1}(M)$ and $\cup_{p\in\mathbb Z} F_p(M)=M$. 
\item[(3)]
$F_p(A)F_q(M)\subseteq F_{p+q}(M)$.
\end{itemize}
If $N$ is an $A$-submodule of $M$, then $N$ is also filtered by 
$F_p(N)=F_p(M)\cap N$. The following is well-known. 

\begin{lemma}
Let $A$ be a filtered algebra. Then an $A$-module $M$ is finitely 
generated if and only if $M$ is filtered 
such that $\gr(M)$ is finitely generated as a $\gr(A)$-module. 
\end{lemma}

The rational Cherednik algebra is a filtered algebra by the filtration
$$
F_p(H_R(\underline\kappa,h))=S(V^*)\otimes_\C S(V)_{\leq p}\otimes_\C RW. 
$$
As $\gr(H_R(\underline\kappa,h))=S(V\oplus V^*)\sharp W$ is a Noetherian $R$-algebra, 
an $H_R(\underline\kappa,h)$-submodule of a finitely generated 
$H_R(\underline\kappa,h)$-module is again finitely generated. 
Hence $\mathcal O$ is an Abelian category, and indecomposable objects in 
$\mathcal O$ are indecomposable $H_R(\underline\kappa,h)$-modules. 
In fact, $\mathcal O$ is a Serre subcategory of $H_R(\underline\kappa,h)\text{-}\Mod$. 

\begin{example}
Let $E\in \Irr W$ and let $I=(e_1,\dots,e_n)$ be the augmentation ideal of $S(V)$. 
Then $(S(V)/I^{r+1}\otimes_\C E)\otimes_\C R$ is an 
$S(V)\sharp W\otimes_\C R$-module. Define
$$
\Delta_r(E)=H_R(\underline\kappa,h)\otimes_{S(V)\sharp W\otimes_\C R}
(S(V)/I^{r+1}\otimes_\C E)\otimes_\C R.
$$
Then $\Delta_r(E)\in\mathcal O$. If $r=0$ we denote it $\Delta(E)$ and 
call them \emph{standard modules}. 
\end{example}

The following values are important. They will determine the highest weight 
category structure of $\mathcal O$. 

\begin{definition}
Let $E\in\Irr W$. Then $z$ acts on $E\otimes_\C R$ by a scalar multiple. 
We denote the value by $c_E\in R$. 
\end{definition}

\subsection{The existence of a progenerator}
As there are some confusions in \cite{GGOR}, we prove the existence theorem 
when the base ring $R$ is a local ring. I do not know whether 
it holds for arbitrary Noetherian ring. One difficulty is that 
the rational Cherednik algebra is not module-finite over $R$.

\begin{lemma}
\label{decomp}
Suppose that $R$ is a local ring such that the residue field $F$ contains $\C$. 
For each $a\in F$, define 
$$
R(a)=\{\alpha\in \cup_{E\in\Irr W}(-c_E+\Z_{\geq0})\mid \bar\alpha=a\}.
$$
Let $f_a(z)=\prod_{\alpha\in R(a)}(z-\alpha)$ be a monic polynomial in $R[z]$. 
\begin{itemize}
\item[(1)]
Let $M\in\mathcal O$. Then $M=\oplus_{a\in F}M_a$ where
$$
M_a=\{m\in M\mid f_a(\Eu)^Nm=0, \text{for sufficiently large $N$}\}.
$$
\item[(2)]
$M\mapsto M_a$ is an exact functor from $\mathcal O$ to $RW\text{-}\Mod$. 
\item[(3)]
$M_a$ is a finitely generated $R$-module. 
\end{itemize}
\end{lemma}

For the proof, consider the case $M=\Delta_r(E)$ first. 

\begin{definition}
For $M\in\mathcal O$, define $M^{prim}=\{m\in M\mid Vm=0\}$.
\end{definition}
Note that $M^{prim}\neq0$ whenever $M\neq0$, and 
$M^{prim}\subseteq \sum_{E\in\Irr W}M_{-\overline{c_E}}$. 

\begin{lemma}
Suppose that $R$ is a Noetherian local ring such that the residue field $F$ 
contains $\C$. Then any $M\in\mathcal O$ is 
a direct sum of finitely many indecomposable objects of $\mathcal O$. 
\end{lemma}
In fact, if we had a strictly increasing infinite sequence of submodules 
$M_1\subseteq M_1\oplus M_2\subseteq\cdots\subseteq M$ then 
we have a strictly increasing sequence
$$
M_1^{prim}\subseteq M_1^{prim}\oplus M_2^{prim}\subseteq\cdots\subseteq M^{prim},
$$
which is a contradiction since $M^{prim}$ is a Noetherian $R$-module. 
Suppose that $R$ is a complete regular local ring. Then 
the uniqueness of the decomposition follows from the existence of a progenerator 
which we will prove in this subsection. However, the case when $R$ is a field 
is an easy case, and we record here the following basic results.
The proof is by standard arguments.

\begin{proposition}[Dunkl-Opdam]
\label{Dunkl-Opdam}
Let $R$ be a field of characteristic $0$.
\begin{itemize}
\item[(1)]
$\Delta(E)$ admits the eigenspace decomposition with respect to $\Eu$. 
Further, the set of eigenvalues that appear in the eigenspace decomposition 
coincides with $-c_E+\Z_{\geq0}$.
\item[(2)]
$\Delta(E)$ has a unique maximal $H_R(\underline\kappa,h)$-submodule 
and the eigenvalue $-c_E$ does not 
appear in its eigenspace decomposition with respect to $\Eu$. 
\item[(3)]
$\End_{\mathcal O}(\Delta(E))=R$, for $E\in\Irr W$. 
\item[(4)]
Let $L(E)=\Top\Delta(E)$, 
which is an irreducible $H_R(\underline\kappa,h)$-module. Then we have
$[\Delta(E):L(E)]=1$ and if $[\Rad\Delta(E):L(E')]\neq0$, for some $E'\in\Irr W$, 
then $-c_{E'}\in -c_E+\Z_{>0}$.
\item[(5)]
$\{L(E)\mid E\in\Irr W\}$ is a complete set of irreducible objects of $\mathcal O$. 
\item[(6)]
Any $M\in\mathcal O$ has a finite length as a $H_R(\underline\kappa,h)$-module.
\item[(7)]
$\mathcal O$ is a Krull-Schmidt category.  
\end{itemize}
\end{proposition}
Recall that an algebra $A$ is local if $A/\Rad A$ is a division ring. 
(7) says that $\End_{\mathcal O}(M)$ is local, for any 
indecomposable object $M$ in $\mathcal O$. In fact, we may apply the Fitting lemma 
to $M$ by (6), and (7) follows. 

\begin{definition}
Let $R$ be a ring which contains $\Z$. 
We introduce a partial order on $R$ by $a\leq b$ if 
$a+\Z_{\geq0}\subseteq b+\Z_{\geq0}$. (Thus, $a\leq a-1$.) 
The order induces a preorder on $\Irr W$ by 
$E_1\leq E_2$ if $-\overline{c_{E_1}}\leq -\overline{c_{E_2}}$. 
\end{definition}

\begin{proposition}
\label{projective object}
Let $R$ be a local ring such that the residue field $F$ contains $\C$. 
For $a\in F$, we define the full subcategory $\mathcal O^{\leq a}$ of 
$\mathcal O$ by
$$
\mathcal O^{\leq a}=
\{M\in\mathcal O \mid M_b=0, \text{for $b\not\leq a$.}\}.
$$
Suppose that $\Delta(E)\in\mathcal O^{\leq a}$ and that
$-\overline{c_{E'}}\not\leq a$ 
when $-\overline{c_{E'}}>-\overline{c_E}$. Then 
$\Delta(E)$ is a projective object of $\mathcal O^{\leq a}$. 
\end{proposition}
\begin{proof}
Let $M\rightarrow N\rightarrow 0$ in $\mathcal O^{\leq a}$ 
and take $0\neq f\in\Hom(\Delta(E),N)$. 
Fix $0\neq v\in E$. Then, to show that 
$\Hom(\Delta(E),M)\rightarrow\Hom(\Delta(E),N)$ is surjective, 
it suffices to prove that there exists $m\in M$ 
such that (i) $m$ maps to $f(v)$, (ii) $RWm\simeq E\otimes_\C R$, 
(iii) $Vm=0$. By Lemma \ref{decomp}(2), we may choose $m\in M_{-\overline{c_E}}$ 
such that $m$ satisfies (i) and (ii). Suppose that $I^rm\neq 0$ and $I^{r+1}m=0$. 
Then
$$
S^r(V)RWm\simeq I^rRWm\subseteq M_{-\overline{c_E}-r}.
$$
Choose $E'\in\Irr W$ 
such that $E'$ appears in $S^r(V)\otimes_\C E$. Then we have
$-\overline{c_{E'}}\leq a$ by $M\in\mathcal O^{\leq a}$ and 
$-\overline{c_{E'}}=-\overline{c_E}-r\geq -\overline{c_E}$.
Thus, $r=0$ by the assumption and (iii) is also satisfied. 
\end{proof}

A similar argument shows the following.

\begin{lemma}
\label{ext}
Suppose that $R$ is a local ring such that the residue field $F$ contains $\C$. 
Let $E, E'\in\Irr W$ be such that $E\not<E'$. Then we have 
$\Ext^1_{\mathcal O}(\Delta(E),\Delta(E'))=0$. 
\end{lemma}

In fact, if $0\rightarrow \Delta(E')\rightarrow M\rightarrow \Delta(E)\rightarrow 0$ 
is given, take $0\neq v\in\Delta(E)_{-\overline{c_E}}$. Then we may choose 
$m\in M_{-\overline{c_E}}$ such that (i) $m$ maps to $v$, (ii) $Vm=0$. 
Hence the exact sequence splits. 

\begin{lemma}
\label{induction}
Suppose that $R$ is a local ring whose residue field $F$ contains $\C$. 
Let $a\in F$ and let $\{\Delta(E_1),\dots,\Delta(E_m)\}$ be all of the standard 
modules that belong to $\mathcal O^{\leq a}$. If we have projective 
objects $P_i$ of $\mathcal O^{\leq a}$ such that 
$P_i\rightarrow \Delta(E_i)\rightarrow 0$, for $1\leq i\leq m$, 
then $P=\oplus_{i=1}^m P_i$ is a progenerator of $\mathcal O^{\leq a}$. 
\end{lemma}

In fact, for any $M\in\mathcal O^{\leq a}$, we have 
$\Delta_r(\C W)^{\oplus N}\rightarrow M\rightarrow 0$, for some $r$ and $N$. 
Note that $\Delta_r(\C W)^{\oplus N}$ has a finite $\Delta$-filtration. 
Suppose that $\Delta(E')$ with $-\overline{c_{E'}}\not\leq a$ appears in the 
$\Delta$-filtration. As any subquotient of $M\in\mathcal O^{\leq a}$ belongs 
to $\mathcal O^{\leq a}$, the image of $\Delta(E')$ vanishes. Hence, 
we have $P^{\oplus N'}\rightarrow M\rightarrow 0$, for some $N'$. 

\begin{theorem}
\label{progenerator}
Suppose that $R$ is a Noetherian local ring whose residue field contains $\C$. 
Then $\mathcal O$ has a progenerator which has a finite $\Delta$-filtration.
\end{theorem}
\begin{proof}
For $\gamma\in F/\Z$, define 
$\mathcal O^\gamma=\{M\in\mathcal O\mid \text{$M_a=0$, for $a\not\in\gamma$.}\}$. 
Then we have $\mathcal O=\oplus_{\gamma\in F/\Z}\mathcal O^\gamma$. 
We show the existence of the desired progenerator for 
each $\mathcal O^\gamma$. For this purpose, we assert that $\mathcal O^\gamma$ 
has a progenerator which is a quotient of an object with 
a finite $\Delta$-filtration. Let $\gamma=a_0+\Z$. If $k$ is sufficiently small 
then $\mathcal O^{\leq a_0+k}=\mathcal O^\gamma$, and if $k$ is sufficiently large then 
$\mathcal O^{\leq a_0+k}=\{0\}$. Hence we prove the assertion by induction on $k$. 
Suppose that we have the desired progenerator $Q$ for $\mathcal O^{<a}$. 
By Proposition \ref{projective object} and Lemma \ref{induction}, it suffices 
to show the existence of a projective object $P$ of $\mathcal O^{\leq a}$ 
such that $P\rightarrow Q\rightarrow 0$ and that $P$ is a quotient 
of an object with a finite $\Delta$-filtration. 
Write 
$Q=\sum_{i=1}^l H_R(\underline\kappa,h)m_i$ such that 
$m_i\in Q_{a_i}$. Fix $N$ so that 
$$
a_i-N-1\not\in \bigcup_{E'\in\Irr W}(-\overline{c_{E'}}+\Z_{\geq0}),
$$
for $1\leq i\leq l$. This implies that $M_{a_i-N-1}=0$, for 
any $M\in\mathcal O$, by Lemma \ref{decomp}. 

Moreover, since $Q$ is a quotient of an object with a finite $\Delta$-filtration and 
$f_{a_i}(\Eu)$ acts as $0$ on $\Delta(E')_{a_i}$, for all $E'\in\Irr W$, 
there exists $e$ such that $f_{a_i}(\Eu)^em=0$, for $m\in Q_{a_i}$. 

\medskip
\noindent
Claim 1: Let $M\in\mathcal O^{\leq a}$. Then 
$f_{a_i}(\Eu)^{e+1}m=0$, for $m\in M_{a_i}$. 

\smallskip
In fact, since $M_a$ is a finitely generated $R$-module, there is 
$$
\varphi:\bigoplus_{E':-\overline{c_{E'}}=a}\Delta(E')^{\oplus m_{E'}}
\longrightarrow M
$$
such that $\Coker\varphi\in\mathcal O^{<a}$. In particular, there exists 
$Q^{\oplus m_Q}\rightarrow \Coker\varphi\rightarrow 0$. 
Hence if $m\in M_{a_i}$ then $f_{a_i}(\Eu)^em\in(\Im\varphi)_{a_i}$ and 
$f_{a_i}(\Eu)^{e+1}m=0$ follows. 

\medskip
Note that $\Delta_N(\C W)$ is the tensor product 
$\mathcal O(V)\otimes_\C \C W\otimes_\C S(V)/I^{N+1}\otimes_\C R$. 
Hence we have $1:=1\otimes 1\otimes 1\otimes 1\in \Delta_N(\C W)$. Define
$$
R_i=\Delta_N(\C W)/H_R(\underline\kappa,h)f_{a_i}(\Eu)^{e+1}1
$$
and define $R_i'$ to be the $H_R(\underline\kappa,h)$-submodule of $R_i$ 
generated by 
$\oplus_{b\not\leq a}(R_i)_b$. Then we let $P_i=R_i/R_i'$. 
Note that $P_i\in\mathcal O^{\leq a}$. We denote the image of $1$ in $P_i$ 
by $1_i$. 

\medskip
\noindent
Claim 2: For $M\in\mathcal O^{\leq a}$, we have a natural isomorphism 
$\Hom_{\mathcal O}(P_i,M)\simeq M_{a_i}$.

\medskip
The isomorphism is given by $\varphi\mapsto \varphi(1_i)$. As 
$P_i=H_R(\underline\kappa,h)1_i$, the injectivity is clear. On the other hand, 
if $m\in M_{a_i}$ then $I^{N+1}m\subseteq M_{a_i-N-1}=0$. Hence 
we have $S(V)/I^{N+1}\rightarrow M$ defined by $1\mapsto m$, which induces 
$\Delta_N(\C W)\rightarrow M$. Since $f_{a_i}(\Eu)^{e+1}m=0$ by Claim 1, 
we obtain $R_i\rightarrow M$. We conclude that there exists 
$\varphi:P_i\rightarrow M$ such that $\varphi(1_i)=m$. 

\medskip
By Lemma \ref{decomp}(2) and Claim 2, $P_i$ is a projective object of 
$\mathcal O^{\leq a}$. $P_i$ is a quotient of $\Delta_N(\C W)$, and $\Delta_N(\C W)$ 
has a finite $\Delta$-filtration. Thus, 
$P=\oplus_{i=1}^l P_i$ has the required properties, and the assertion is 
proved. 

Now we have a progenerator of $\mathcal O^\gamma$ which is a quotient of 
an object with a finite $\Delta$-filtration. 
That this progenerator has a finite $\Delta$-filtration comes from 
the following claim, which is not difficult to prove. 

\medskip
\noindent
Claim 3: If $M_1\oplus M_2$ has a finite 
$\Delta$-filtration, then so does $M_1$ and $M_2$. 
\end{proof}

\begin{proposition}
\label{key prop}
Suppose that $R$ is regular. That is, $R$ is Noetherian and $R_{\mathfrak p}$ 
is a regular local ring, for all $\mathfrak p\in\Spec R$. 
Let $H$ be an $R$-algebra, $\mathcal C$ 
a full subcategory of $H\text{-}\fgMod$. Suppose that if $M\in\mathcal C$ and 
$M\rightarrow N\rightarrow 0$ in $H\text{-}\fgMod$ then $N\in\mathcal C$,  
and that there exists a projective object $P$ of $\mathcal C$ such that 
\begin{itemize}
\item[(i)]
$P$ is a finitely presented $H$-module,
\item[(ii)]
$P$ is a projective $R$-module,
\item[(iii)]
$\End_{\mathcal C}(P)$ is a finitely generated $R$-module.
\end{itemize}
Then $A=\End_{\mathcal C}(P)^{\rm op}$ is a projective $R$-module. 
If $Q\in\mathcal C$ satisfies
\begin{itemize}
\item[(a)]
$Q$ is a projective $R$-module,
\item[(b)]
$\Hom_{\mathcal C}(P,Q)$ is a finitely generated $R$-module,
\end{itemize}
then $\Hom_{\mathcal C}(P,Q)$ is a projective $R$-module.
\end{proposition}
\begin{proof}
As $R$ is Noetherian and by (b), $\Hom_{\mathcal C}(P,Q)$ is a 
finitely presented $R$-module. Hence 
$$
\Ext_R^1(\Hom_{\mathcal C}(P,Q),?)_{\mathfrak p}=
\Ext^1_{R_{\mathfrak p}}(\Hom_{\mathcal C}(P,Q)_{\mathfrak p},
?_{\mathfrak p}).
$$
We also have $\Hom_{\mathcal C}(P,Q)_{\mathfrak p}\simeq 
\Hom_{H_{\mathfrak p}}(P_{\mathfrak p},Q_{\mathfrak p})$ by (i). 
Hence, $\Hom_{\mathcal C}(P,Q)$ is a projective $R$-module if and only if 
$\Hom_{H_{\mathfrak p}}(P_{\mathfrak p},Q_{\mathfrak p})$ is a free 
$R_{\mathfrak p}$-module. Hence we may assume that $R$ is a regular local ring 
from the beginning. Let $d$ be the Krull dimension of $R$, $(z_1,\dots,z_d)$ 
be a system of parameters. Let $\mathcal F=\Hom_{\mathcal C}(P,-)$ and define 
$Q_i=\mathcal F(Q/(z_1,\dots,z_i)Q)$. By (a), we have the exact sequence
$$
0\rightarrow Q/(z_1,\dots,z_i)Q\overset{z_{i+1}\cdot}{\rightarrow}
Q/(z_1,\dots,z_i)Q\rightarrow Q/(z_1,\dots,z_{i+1})Q\rightarrow 0.
$$
Thus, $0\rightarrow Q_i\overset{z_{i+1}\cdot}{\rightarrow}
Q_i\rightarrow Q_{i+1}\rightarrow 0$. In particular, all the $Q_i$ are finitely 
generated $R$-modules since $Q_0$ is so by (b). Let $R_i=R/(z_1,\dots,z_i)$. 
We claim that $Q_i$ is a flat $R_i$-module. In fact, this is obvious when 
$i=d$. Suppose that $i<d$ and that $Q_{i+1}$ is a flat $R_{i+1}$-module. 
Then, since 
\begin{itemize}
\item[(i)]
$Q_i$ is a finitely generated module over the Noetherian local ring $R_i$, 
\item[(ii)]
$z_{i+1}\in R_i$ is a non-invertible regular element, 
\end{itemize}
$Q_i$ is a flat $R_i$-module if and only if 
$0\rightarrow Q_i\overset{z_{i+1}\cdot}{\rightarrow}Q_i$ and 
$Q_{i+1}\simeq Q_i/z_{i+1}Q_i$ is a flat $R_{i+1}$-module. Hence,  
$Q_0$ is a flat $R$-module, which is a free $R$-module by (b). 
\end{proof}

\begin{theorem}
\label{reduction to A-mod}
Suppose that $R$ is a regular local ring whose residue field $F=R/\mathfrak m$ 
contains $\C$. 
We denote by $\mathcal O(R)$ the category $\mathcal O$ for 
$H_R(\underline\kappa,h)$, and we identify 
$\mathcal O(F)$, the category $\mathcal O$ for $H_F(\underline\kappa,h)$, 
with the full subcategory of $\mathcal O(R)$ consisting of 
$M$ with $\mathfrak m M=0$. 
Then, there exists a module-finite $R$-algebra $A$ such that we have 
\begin{itemize}
\item[(1)]
$A$ is a projective $R$-algebra, i.e. $A$ is projective as an $R$-module,
\item[(2)]
$\mathcal O(R)\simeq A\text{-}\fgMod$,
\item[(3)]
$\mathcal O(F)$ corresponds to 
$A\otimes_R F\text{-}\fgMod$ under the equivalence in (2),
\item[(4)]
$\Delta(E)$ corresponds to an $A$-module which is a finitely generated projective 
$R$-module under the equivalence. 
\end{itemize}
\end{theorem}

For the proof, we take the progenerator $P$ constructed in Theorem \ref{progenerator}, 
and define $A=\End_{\mathcal O}(P)^{\rm op}$. 
Then we check the assumptions in Proposition \ref{key prop}. (ii) and (a) are obvious. 
Recall that if $0\rightarrow L\rightarrow M\rightarrow N\rightarrow 0$ such that 
$L$ is finitely generated and $M$ is finitely presented then $N$ is finitely 
presented. Hence, that $\Delta_N(\C W)$ is a finitely presented 
$H_R(\underline\kappa,h)$-module implies that (i) holds. (iii) and (b) follow 
from the fact that $P_a$ and $Q_a$ are finitely generated $R$-modules,  
for $a\in F$, since $R$ is Noetherian. 

\begin{corollary}
Suppose that $R$ is a complete regular local ring. Then the category $\mathcal O$ 
is a Krull-Schmidt category. 
\end{corollary}

Suppose that $R$ is a Noetherian commutative ring and $A$ is a module-finite projective 
$R$-algebra. Then a finitely generated $A$-module $M$ is projective if and only if 
the following hold. 
\begin{itemize}
\item[(i)]
$M$ is a projective $R$-module,
\item[(ii)]
for each closed point $x\in\Spec R$, 
$M\otimes_R k(x)$ is a projective $A\otimes_R k(x)$-module.
\end{itemize}
By Theorem \ref{reduction to A-mod}, if $M$ is $\Delta$-filtered 
then $\Hom_{\mathcal O}(P,M)\otimes_R F\simeq\Hom_{\mathcal O}(P,M\otimes_R F)$. 
Hence, we have the following corollary.

\begin{corollary}
Let $R$ be a regular local ring whose residue field $F$ 
contains $\C$. Then a $\Delta$-filtered object $M\in\mathcal O(R)$ is a 
projective object if and only if $M\otimes_R F$ is a projective object of 
$\mathcal O(F)$. 
\end{corollary}

\subsection{Highest weight category}
The aim of the remaining part is to introduce Rouquier's theory of 
quasihereditary covers. 

\begin{definition}
Let $R$ be a commutative ring, $H$ an $R$-algebra, 
$\mathcal C$ an $R$-linear Abelian category which is a 
subcategory of $H\text{-}\Mod$. 
Let $\Lambda$ be a finite preordered set. Then, we say that 
$(\mathcal C,\Lambda)$ is a 
\emph{highest weight category in weak sense} if there exist objects 
$\{\Delta(\lambda)\mid \lambda\in\Lambda\}$ such that the following are satisfied. 
\begin{itemize}
\item[(i)]
$\Delta(\lambda)$ is a projective $R$-module.
\item[(ii)]
If $\Hom_{\mathcal C}(\Delta(\lambda'),\Delta(\lambda''))\neq0$ then 
$\lambda'\leq\lambda''$. 
\item[(iii)]
If $N\in\mathcal C$ is such that $\Hom_{\mathcal C}(\Delta(\lambda),N)=0$, 
for all $\lambda$, then $N=0$. 
\item[(iv)]
For each $\lambda$, there exists a projective object $P(\lambda)$ of $\mathcal C$ 
such that there is $P(\lambda)\rightarrow \Delta(\lambda)\rightarrow 0$ and 
that $\Ker(P(\lambda)\rightarrow \Delta(\lambda))$ 
has a finite $\Delta$-filtration in which only $\Delta(\lambda')$ with 
$\lambda'>\lambda$ appear.
\end{itemize}
A highest weight category in weak sense 
$(\mathcal C,\Lambda)$ is \emph{split} if it also satisfies 
\begin{itemize}
\item[(v)]
$\End_{\mathcal C}(\Delta(\lambda))=R$, for all $\lambda$.
\end{itemize}
\end{definition}

$\Delta(\lambda)$ are called \emph{standard objects}. 
This definition drops the requirement that a highest weight category should be  
Artin in some sense. Hence, we add the phrase \lq\lq in weak sense\rq\rq. 
Following \cite{R}, we define as follows. 

\begin{definition}
A highest weight category in weak sense $(\mathcal C,\Lambda)$ is a 
\emph{highest weight category} if $\mathcal C\simeq A\text{-}\fgMod$, 
for some module-finite projective $R$-algebra $A$.
\end{definition}

Let us recall the usual definition of a highest weight 
category over a field. Note that we only require that 
$\Lambda$ is preordered. But this is not essential. 

\begin{definition}
Let $R$ be a field. A category $\mathcal C$ is an \emph{Artin category} over $R$ if 
\begin{itemize}
\item[(i)]
$\mathcal C$ is an Abelian $R$-linear category,
\item[(ii)]
$\dim_R\Hom_{\mathcal C}(X,Y)<\infty$, for all $X,Y\in\mathcal C$,
\item[(iii)]
All objects are of finite length.
\end{itemize}
A highest weight category $(\mathcal C,\Lambda)$ is \emph{split} 
if it also satisfies 
\begin{itemize}
\item[(v)]
$\End_{\mathcal C}(\Delta(\lambda))=R$, for all $\lambda$.
\end{itemize}
\end{definition}

\begin{definition}
Let $R$ be a field, $\mathcal C$ an Artin category over $R$, 
$\Lambda$ a finite preordered set. Then, we say that 
$(\mathcal C,\Lambda)$ is a 
\emph{highest weight category} if there exist objects 
$\{\Delta(\lambda)\mid \lambda\in\Lambda\}$ such that the following are satisfied. 
\begin{itemize}
\item[(i)]
$L(\lambda)=\Top\Delta(\lambda)$ is an irreducible object and 
$\{L(\lambda)\mid \lambda\in\Lambda\}$ is a complete set of isomorphism classes of 
irreducible objects. 
\item[(ii)]
If $[\Rad\Delta(\lambda):L(\mu)]\neq0$ then $\mu<\lambda$. 
\item[(iii)]
For each $\lambda$, there exists a projective object $P(\lambda)$ of $\mathcal C$ 
such that there is $P(\lambda)\rightarrow \Delta(\lambda)\rightarrow 0$ and 
that $\Ker(P(\lambda)\rightarrow \Delta(\lambda))$ 
has a finite $\Delta$-filtration in which only $\Delta(\lambda')$ with 
$\lambda'>\lambda$ appear.
\end{itemize}
\end{definition}

Suppose that $(\mathcal C,\Lambda)$ is a highest weight category in weak sense. 
If $\lambda'\not<\lambda''$ then 
$\Ext_{\mathcal C}^1(\Delta(\lambda'),\Delta(\lambda''))=0$. Next suppose that 
$(\mathcal C,\Lambda)$ is a highest weight category over a field. 
If $\lambda'\neq\lambda''$ and $\lambda'\not<\lambda''$ then 
$\Hom_{\mathcal C}(\Delta(\lambda'),\Delta(\lambda''))=0$. If 
$(\mathcal C,\Lambda)$ is a highest weight category over a regular local ring 
$R$ whose residue field is $F$, then repeated use of Nakayama's lemma implies that 
if $A\otimes_R F\text{-}\fgMod$ is a highest weight category whose standard objects 
are $\{\Delta(\lambda)\otimes_R F\mid\lambda\in\Lambda\}$ and all $\Delta(\lambda)$ 
are projective $R$-modules, then 
$\lambda'\neq\lambda''$ and $\lambda'\not<\lambda''$ imply 
$\Hom_{\mathcal C}(\Delta(\lambda'),\Delta(\lambda''))=0$. 

\begin{lemma}
Suppose that $R$ is a field. Then the two definitions of split 
highest weight category coincide.
\end{lemma}
That the usual definition implies Rouquier's definition is clear. Hence 
we prove the converse. If $A$ is module-finite over a field $R$, it is a finite 
dimensional $R$-algebra, so $A\text{-}\fgMod$ is automatically 
an Artin category over $R$. By the conditions (ii) and (iv) we have 
$\Hom_{\mathcal C}(P(\lambda),\Delta(\lambda))=\End_{\mathcal C}(\Delta(\lambda))$. 
If $\Top\Delta(\lambda)$ is not irreducible, then 
$\dim_R\End_{\mathcal C}(\Delta(\lambda))\geq2$. Thus, 
$L(\lambda)=\Top\Delta(\lambda)$ is irreducible. Now (i) is clear. 
If $[\Rad\Delta(\lambda):L(\mu)]\neq0$ then we have a nonzero 
$\varphi:P(\mu)\rightarrow \Rad\Delta(\lambda)$. Consider the 
$\Delta$-filtration $P(\mu)=F_0\supseteq F_1=\Ker(P(\mu)\rightarrow\Delta(\mu))$. 
If $\varphi(F_1)=0$ then $\varphi$ induces a nonzero 
$\Delta(\lambda)\rightarrow\Rad\Delta(\lambda)$, which contradicts 
$\dim_R\End_{\mathcal C}(\Delta(\lambda))=1$. Hence,  
there exists $\nu>\mu$ such that $\varphi$ induces a nonzero homomorphism 
$\Delta(\nu)\rightarrow\Delta(\lambda)$. Thus, $\lambda\geq\nu>\mu$ and (ii) is proved. 

\begin{theorem}[Guay]
\label{Guay}
If $R$ is a field then the category $\mathcal O$ for 
$H_R(\underline\kappa,h)$ is a split highest weight category. 
\end{theorem}
In fact, by the argument which uses eigenvalues of $\Eu$,
$\Hom_{\mathcal O}(X,Y)$ is a finitely generated $R$-module, for $X, Y\in\mathcal O$,
when $R$ is a Noetherian local ring. Hence $\mathcal O$ is Hom-finite.
The other conditions but the existence of projective objects $P(E)$ are 
already proved in Proposition \ref{Dunkl-Opdam}. 
By Theorem \ref{reduction to A-mod}, $\Delta(E)$ has the projective cover, 
which we denote by $P(E)$. Then, $P(E)$ is a direct summand of a 
$\Delta$-filtered object. Thus $P(E)$ is $\Delta$-filtered. 
Lemma \ref{ext} implies that any $M\in\mathcal O^\gamma$ with a finite 
$\Delta$-filtration admits a $\Delta$-filtration 
$M=F_0\supseteq F_1\supseteq\cdots$ with the property that 
$F_i/F_{i+1}=\Delta(E')$ and $F_{i+1}/F_{i+2}=\Delta(E'')$ then 
$E'\leq E''$. Thus, if the $\Delta$-filtration 
$$
\Ker(P(E)\rightarrow\Delta(E))=F_1\supseteq F_2\supseteq\cdots\supseteq F_k
\supseteq\cdots\cdots
$$ 
has the form $F_i/F_{i+1}=\Delta(E^{(i)})$ with $E^{(i)}>E$, for $1\leq i<k-1$, 
and $F_{k-1}/F_k=\Delta(E')$ with $E'\not>E$, then $E'\leq E$ and 
we may move $\Delta(E')$ to the top of the filtration. This implies that 
$\Top P(E)\supseteq L(E)\oplus L(E')$, which is a contradiction. 

\begin{definition}
Let $R$ be a commutative ring, $H$ an $R$-algebra, 
$\mathcal C$ an $R$-linear Abelian category which is a 
subcategory of $H\text{-}\Mod$.  
An object $L$ of $\mathcal C$ is \emph{$R$-split} if 
the canonical $R$-module homomorphism 
$$
L\otimes_{\End_A(L)} \Hom_{\mathcal C}(L,P)\rightarrow P
$$
is a split monomorphism in $R\text{-}\Mod$, for all projective objects 
$P$ of $\mathcal C$. 
\end{definition}

It is not named but the definition is in \cite{R}.
The next lemma is from \cite[4.10]{R}.

\begin{lemma}
\label{R-split}
Let $R$ be a Noetherian local ring whose residue field is $F=R/\mathfrak m$, 
$A$ a module-finite projective $R$-algebra, 
$\mathcal C(R)=A\text{-}\fgMod$. We denote the full subcategory 
$\{M\in\mathcal C\mid \mathfrak mM=0\}$ by $\mathcal C(F)=A\otimes_R F\text{-}\fgMod$. 
Let $L\in\mathcal C(R)$. Then, $L$ is an $R$-split projective object of 
$\mathcal C(R)$ if and only if
\begin{itemize}
\item[(i)]
$L$ is a projective $R$-module, 
\item[(ii)]
$L\otimes_R F$ is an $F$-split projective object of $\mathcal C(F)$. 
\end{itemize}
\end{lemma}

Suppose (i) and (ii). Then, by our assumptions on $A$ and $R$, 
$L$ is a projective object of $\mathcal C(R)$. Now, both $A$ and $L$ are 
free $R$-modules of finite rank and $L$ is a direct summand of a free $A$-module
of finite rank. Thus 
\begin{gather*}
\Hom_A(L,A)\otimes_R F=\Hom_{A\otimes_R F}(L\otimes_R F,A\otimes_R F),\\
\End_A(L)\otimes_R F=\End_{A\otimes_R F}(L\otimes_R F).
\end{gather*}
Since $L\otimes_R F$ is $F$-split, 
$L\otimes_{\End_A(L)}\Hom_A(L,A)\otimes_R F\rightarrow A\otimes_R F$ is a monomorphism.
Thus if we write
$0\rightarrow L\otimes_{\End_A(L)}\Hom_A(L,A)\rightarrow A\rightarrow L'\rightarrow 0$
then $\Tor_1^R(F,L')=0$.  
Hence, $L'$ is flat by \cite[Theorem 22.3]{Mat}, 
and $L$ is $R$-split as desired. The other implication is clear. 

\begin{lemma}
\label{F-split to R-split}
Let $R$ be a regular local ring whose residue field is $F$, 
$A$ a module-finite projective $R$-algebra, 
$\mathcal C(R)=A\text{-}\fgMod$ and $\mathcal C(F)=A\otimes_R F\text{-}\fgMod$. 
Suppose that a collection of $R$-free objects 
$\{\Delta(\lambda)\mid \lambda\in\Lambda\}\subseteq\mathcal C(R)$ 
is given. If $\mathcal C(F)$ is a split highest weight category whose standard 
objects are $\{\Delta(\lambda)\otimes_RF\mid \lambda\in\Lambda\}$, and whose 
projective objects are $\Delta\otimes_R F$-filtered, 
then $\Delta(\lambda)$ is $R$-split, for all $\lambda$. 
\end{lemma}

As we only apply the result to $\mathcal O$, we add the assumption that 
projective objects are $\Delta$-filtered, but this is in fact automatic. 

Let $\lambda$ be a maximal element of $\Lambda$, and let 
$P$ be a projective object of $\mathcal C(F)$. Then, since $\mathcal C(F)$ is 
a split highest weight category, 
$\Delta(\lambda)\otimes_R F$ is a projective object of $\mathcal C(F)$, 
there is a subobject $P_0$ of $P$ such that 
$P_0\simeq(\Delta(\lambda)\otimes_R F)^{\oplus m}$, for some $m$, 
and $\Hom_{\mathcal C(F)}(\Delta(\lambda)\otimes_R F,P/P_0)=0$. 
We also have $\End_{\mathcal C(F)}(\Delta(\lambda)\otimes_R F)=F$. 
Thus, $\Delta(\lambda)\otimes_R F$ is an $F$-split projective object of 
$\mathcal C(F)$ by \cite[Lemma 4.5]{R}. As $\Delta(\lambda)$ is $R$-free, 
it is an $R$-split projective object of $\mathcal C(R)$ 
by Lemma \ref{R-split}. Define 
$$
J=\Im\left(\Delta(\lambda)\otimes_R\Hom_A(\Delta(\lambda),A)
\rightarrow A\right).
$$
$J$ is a two-sided ideal of $A$. As $\Delta(\lambda)$ is 
$R$-split and projective, \cite[Lemma 4.5]{R} implies that 
$A/J$ is a module-finite projective $R$-algebra, $\Hom_A(\Delta(\lambda),A/J)=0$ and 
$J\simeq \Delta(\lambda)^{\oplus m}$, for some $m$. 
Thus, we may prove Lemma \ref{F-split to R-split} by induction on $|\Lambda|$. 
Note that we need here 
$\Hom_{\mathcal C}(\Delta(\lambda'),\Delta(\lambda''))=0$, if 
$\lambda'\neq\lambda''$ and $\lambda'\not<\lambda''$, to guarantee that 
if $\lambda'\neq\lambda$ then $\Delta(\lambda')$ is an $A/J$-module. 
See \cite[Lemma 4.4]{R}. 

Now we have Theorem \ref{split hw cat} below.
The conditions (i), (ii) and (v) to
be a split highest weight category are easy to check. 
By Theorem \ref{reduction to A-mod}, 
we assume that objects of $\mathcal O(R)$ are finitely generated 
$R$-modules. Thus, $N=0$ if and only if $N\otimes_R F=0$. Hence 
Theorem \ref{Guay} implies that (iii) holds. 
To verify (iv) by induction, we show that if $\lambda$ is maximal 
and $Q$ is a finitely generated projective $A/J$-module such that some 
$\Delta(\mu)$ with $\mu<\lambda$ appears in its $\Delta$-filtration, 
then we may find a finitely generated projective 
$A$-module $P$ such that
$$
0\rightarrow \Delta(\lambda)^{\oplus m}\rightarrow P\rightarrow 
Q\rightarrow 0,\;\;\text{for some $m$.}
$$
This is proved in \cite[Lemma 4.9]{R}. Take a surjective map 
$f:R^{\oplus m}\rightarrow\Ext_A^1(Q,\Delta(\lambda))$. 
Then $f\in\Hom_R(R^{\oplus m},\Ext_A^1(Q,\Delta(\lambda)))\simeq 
\Ext_A^1(Q,\Delta(\lambda)^{\oplus m})$ gives the desired short 
exact sequence. 
Note that if $Q=Q_1\oplus Q_2$ then we consider the direct sum of 
the short exact sequences for $Q_1$ and $Q_2$, so we may prove the assertion that $P$ 
is a projective $A$-module only when $Q=(A/J)^{\oplus n}$, for some $n$. 

\begin{theorem}[Rouquier]
\label{split hw cat}
Let $R$ be a regular local ring whose residue field $F$ contains $\C$. 
Then the category $\mathcal O(R)$ is a split highest weight category.
\end{theorem}

\subsection{The Knizhnik-Zamolodchikov functor} 
In this subsection we consider the regular local ring
$$
R=\C[[\bk_1-\kappa_1,\dots,\bk_{d-1}-\kappa_{d-1},\bh-h]], 
$$
where $\bk_1,\dots,\bk_{d-1},\bh$ are 
indeterminates and $\kappa_1,\dots,\kappa_{d-1},h\in\C$. 
Let $H_R(\underline\bk,\bh)$ be the rational Cherednik algebra. Then 
$H_R(\underline\bk,\bh)\otimes_R\C=H_\C(\underline\kappa,h)$. We denote 
the category $\mathcal O$ for $H_R(\underline\bk,\bh)$ by $\mathcal O(R)$, 
and the category $\mathcal O$ for $H_\C(\underline\kappa,h)$ by 
$\mathcal O$ itself. 

\begin{definition}
For $M\in\mathcal O(R)$, we denote the sheaf 
$(\mathcal O_{V_{reg}}\otimes_\C R)\otimes_{\mathcal O(V)\otimes_\C R}M$ on 
$V_{reg}$ by $\tilde M$. Similarly, we define $\tilde M$, for $M\in\mathcal O$. 
\end{definition}

Let $X$ be a connected smooth algebraic variety over $\C$, let $\mathcal M$ be 
a free $\mathcal O_X$-module of finite rank with a flat connection $\nabla$. 
Then we say that $\mathcal M$ is \emph{regular} if 
$p^*\mathcal M=\mathcal D_{C\rightarrow X}\otimes_{p^{-1}\mathcal D_X}p^{-1}\mathcal M$, 
where $\mathcal D_{C\rightarrow X}=
\mathcal O_C\otimes_{p^{-1}\mathcal O_X}p^{-1}\mathcal D_X$, 
for any smooth curve $C$ and $p:C\rightarrow X$, has regular singularities. 

Let $X=V_{reg}$. Then $\mathcal D_X\sharp W\otimes_\C R$ is a localization 
of $H_R(\underline\bk,\bh)$ with respect to $\mathcal O(X)$. For each 
$M\in\mathcal O(R)$, we may consider $\tilde M$ as a 
$\mathcal D_X\sharp W\otimes_\C R$-module via
$$
\nabla_{\partial_y}=y-\sum_{H\in\mathcal A}\frac{\langle \alpha_H,y\rangle}{\alpha_H}a_H.
$$
Similarly, $\tilde M$ is a $\mathcal D_X\sharp W$-module, for $M\in\mathcal O$. 

\begin{example}
If $M=\Delta(E)=R[x_1,\dots,x_n]\otimes_\C E$, then, for $p\in R[x_1,\dots,x_n]$ and 
$v\in E$, we have 
\begin{equation*}
\begin{split}
\nabla_{\partial_y}(p\otimes v)&=
(\partial_yp)\otimes v-\sum_{i=1}^n\sum_{k=0}^{d-1}\kappa_ky_i\frac{p}{x_i}\otimes 
d\epsilon_{H_i,k}v\\
&\quad\quad-\sum_{1\leq i\leq j\leq n}\sum_{\alpha=0}^{d-1}h(y_i-\zeta^\alpha y_j)
\frac{p}{x_i-\zeta^\alpha x_j}\otimes(1-s_{ij;\alpha})v.
\end{split}
\end{equation*}
\end{example}

It is easy to see that any $\Delta(E)\in\mathcal O$ gives a regular holonomic 
$\mathcal D_X\sharp W$-module whose characteristic variety is $T^*_XX$. 
Hence, $\tilde M$ is regular, for all $M\in\mathcal O$. 

By Deligne's Riemann-Hilbert correspondence, the category of free 
$\mathcal O_X$-modules of finite rank with a flat connection is 
equivalent to the category of $\C_X$-modules of finite rank. 
Hence, $M\in\mathcal O$ defines a $\C_X$-module 
$\mathcal Hom_{\mathcal D_{X^{an}}}
(\mathcal O_{X^{an}},\mathcal O_{X^{an}}\otimes_{\mathcal O_X}\tilde M)$. 
Here, the $\mathcal O_{X^{an}}$ on the left hand side is the 
$\mathcal D_{X^{an}}$-module 
$\mathcal D_{X^{an}}/\sum_{i=1}^n\mathcal D_{X^{an}}\frac{\partial}{\partial x_i}$. 
The $\C_X$-module is nothing but the sheaf of horizontal sections of 
$\mathcal O_{X^{an}}\otimes_{\mathcal O_X}\tilde M$. 
As our connection satisfies  
$w\nabla_{\partial_y}w^{-1}=\nabla_{\partial_{w(y)}}$, for $w\in W$, 
it is in the category of $W$-equivariant $\C_X$-modules 
of finite rank. The latter is equivalent to 
the category of $\C_{X/W}$-modules of finite rank, and, by taking the monodromy, 
it is equivalent to the category of finite dimensional 
$\C\pi_1(X/W)$-modules. Let us denote it by $\C\pi_1(X/W)\text{-}\fgMod$. 
Hence, we have obtained an exact functor
$$
KZ:\mathcal O\longrightarrow \C\pi_1(X/W)\text{-}\fgMod.
$$

The following is proved in \cite[Proposition 5.9]{GGOR}.

\begin{lemma}
\label{fully faithful}
Suppose that $v_i=\exp(2\pi\sqrt{-1}(\kappa_i+\frac{i}{d}))$, 
$1\leq i\leq d-1$, are different from 
$1$ and pairwise distinct, $q=\exp(2\pi\sqrt{-1}h)\neq-1$. 
Let $M, N\in\mathcal O$ and suppose that $N$ is $\Delta$-filtered. 
Then 
$\Hom_{\mathcal O}(M,N)\simeq\Hom_{\C\pi_1(X/W)}(KZ(M),KZ(N))$. 
In particular, the KZ functor is fully faithful 
on the additive subcategory of $\Delta$-filtered objects.
\end{lemma}

Now, we consider $\mathcal O(R)$ and the functor
$$
KZ(R):\mathcal O(R)\longrightarrow R\pi_1(X/W)\text{-}\fgMod
$$
which is defined in the same way as above. By \cite[Theorem 2.23]{De}, 
$KZ(R)$ is an exact functor. 

\begin{lemma}
\label{KZ image}
If $M\in\mathcal O(R)$ is $\Delta$-filtered, then 
$\mathcal Hom_{\mathcal D_{X^{an}}}(\mathcal O_{X^{an}},\tilde M)$ is a locally constant 
free $\C_X\otimes_\C R$-module of 
rank $\sum_{E\in\Irr W}[M:\Delta(E)]\dim_\C E$. 
\end{lemma}

To show this, set $\mathbf t_i=\bk_i-\kappa_i$, for $1\leq i<d$, and 
$\mathbf t_d=\bh-h$, then, we write 
$p=\sum_{\underline n} p_{\underline n}\mathbf t^{\underline n}\in\tilde M$, where 
$p_{\underline n}\in\mathcal O_X$, and solve $\nabla_{\partial_y}p=0$ by 
recursively solving the system of equations for $p_{\underline n}$. 

Let $K=\C((\bk_1-\kappa_1,\dots,\bk_{d-1}-\kappa_{d-1},\bh-h))$ and 
denote by $\mathcal O(K)$ the category $\mathcal O$ for 
$H_K(\underline\bk,\bh)$. We have
$KZ(K):\mathcal O(K)\longrightarrow K\pi_1(X/W)\text{-}\fgMod$.
Results in \cite{BMR} imply the following.

\begin{theorem}[Brou\'e-Malle-Rouquier]
There is an explicit choice of generators $\sigma_0,\dots,\sigma_{n-1}$ of 
$\pi_1(X/W)$ such that the defining relations are given by 
$$
(\sigma_0\sigma_1)^2=(\sigma_1\sigma_0)^2,\;\;
\sigma_i\sigma_{i+1}\sigma_i=\sigma_{i+1}\sigma_i\sigma_{i+1}\;(i\neq0),\;\;
\sigma_i\sigma_j=\sigma_j\sigma_i\;(j\geq i+2).
$$
Further, $KZ(K)(M)$, for $M\in\mathcal O(K)$, factors through
$$
(\sigma_0-1)(\sigma_0-\mathbf v_1)\cdots(\sigma_0-\mathbf v_{d-1})=0,\;\;
(\sigma_i-1)(\sigma_i+\mathbf q)=0,\;\text{if $i\neq0$,}
$$
where $\mathbf v_i=\exp(2\pi\sqrt{-1}(\bk_i+\frac{i}{d}))$, 
for $1\leq i\leq d-1$, and $\mathbf q=\exp(2\pi\sqrt{-1}\bh)$. 
\end{theorem}

For the proof of the second part, consider the KZ functor over 
$\C[\underline\bk,\bh]$, which we denote by $KZ(\C[\underline\bk,\bh])$. 
It is an exact functor by \cite[Theorem 2.23]{De}. 
If $(\underline\kappa,h)$ is in certain open dense subset of 
generic enough parameters, it is proved 
in \cite{BMR} that 
$KZ(\Delta(\C W))$ factors through the Hecke algebra. Now, we have 
$$
E^2_{pq}=\Tor_p^{\C[\underline\bk,\bh]}(R^{-q}KZ(\C[\underline\bk,\bh])(M),\C)\Longrightarrow 
R^{p+q}KZ(M\otimes_{\C[\underline\bk,\bh]}\C),
$$
if $M\in\mathcal O(\C[\underline\bk,\bh])$ is a flat $\C[\underline\bk,\bh]$-module, 
by the K\"unneth spectral sequence. $E^2_{pq}=0$ if $q\neq 0$. Let  
$M=\Delta(\C W)\otimes_\C\C[\underline\bk,\bh]$. Then 
$KZ(\C[\underline\bk,\bh])(M)$ is also a flat $\C[\underline\bk,\bh]$-module 
by Lemma \ref{KZ image}. 
Hence
$$
KZ(\C[\underline\bk,\bh])(\Delta(\C W)\otimes_\C\C[\underline\bk,\bh])
\otimes_{\C[\underline\bk,\bh]}\C\simeq 
KZ(\Delta(\C W)),
$$
and the right hand side factors through the Hecke algebra for the generic enough 
$(\underline\kappa,h)$'s. Thus, $KZ(K)(\Delta(\C W)\otimes_\C K)$ 
factors through the Hecke algebra. 
Now, $\mathcal O(K)$ is a semisimple category whose irreducible objects are 
standard objects, which are direct summands of $\Delta(\C W)\otimes_\C K$. 
Hence, the result follows. 

\begin{definition}
We denote by $\H_n(\underline{\mathbf v},\mathbf q)$ 
the quotient of $R\pi_1(X/W)$ by the two-sided ideal 
generated by $(\sigma_0-1)(\sigma_0-\mathbf v_1)\cdots(\sigma_0-\mathbf v_{d-1})$ 
and $(\sigma_i-1)(\sigma_i+\mathbf q)$, for $i\neq0$. Note that if 
$\mathbf v_i=\mathbf q^{\gamma_i}$, for $1\leq i\leq d-1$, then 
$\H_n(\underline{\mathbf v},\mathbf q)=\H_n^\Lambda(\mathbf q)$ with 
$\Lambda=\Lambda_0+\sum_{i=1}^{d-1}\Lambda_{\gamma_i}$. 
\end{definition}

\begin{lemma}
\label{representable}
Let $R$ be a Noetherian commutative ring. Let 
$A$ be a module-finite projective $R$-algebra, $B$ an $R$-algebra. 
Suppose that there exists an exact functor 
$\mathcal F:A\text{-}\fgMod\rightarrow B\text{-}\fgMod$ such that 
$\mathcal F(A)$ is a finitely generated projective $R$-module. Then, 
$\mathcal F\simeq\Hom_A(P,-)$, where the finitely generated projective 
$A$-module $P$ is given by the $(A,B)$-bimodule $P=\Hom_A(\mathcal F(A),A_A)$. 
\end{lemma}

Note that $\Hom_A(P,A)=\mathcal F(A)$. Thus, 
$\mathcal G=\Hom_B(\mathcal F(A),-)$ gives the right adjoint of $\mathcal F$. 
If $B=\End_A(P)^{\rm op}$ then $\mathcal F\mathcal G(M)\simeq M$, for 
$M\in B\text{-}\fgMod$. 

The starting point of Rouquier's theory of quasihereditary 
covers is the theorem below, and we are reduced
to a purely algebraic setting.

\begin{theorem}
Let $R=\C[[\bk_1-\kappa_1,\dots,\bk_{d-1}-\kappa_{d-1},\bh-h]]$ and 
suppose that $v_i=\exp(2\pi\sqrt{-1}(\kappa_i+\frac{i}{d}))$, 
$1\leq i\leq d-1$, are different from 
$1$ and pairwise distinct, $q=\exp(2\pi\sqrt{-1}h)\neq-1$. 
Then there is a projective 
object $P_{KZ}\in\mathcal O(R)$ such that 
\begin{itemize}
\item[(1)]
$KZ(R)\simeq\Hom_{\mathcal O(R)}(P_{KZ},-):\mathcal O(R)\rightarrow 
\H_n(\underline{\mathbf v},\mathbf q)\text{-}\fgMod$, and it induces 
$KZ\simeq\Hom_{\mathcal O}(P_{KZ}\otimes_R\C,-):\mathcal O\rightarrow 
\H_n(\underline v,q)\text{-}\fgMod$.
\item[(2)]
$\H_n(\underline{\mathbf v},\mathbf q)\simeq\End_{\mathcal O(R)}(P_{KZ})^{\rm op}$ and 
$\H_n(\underline v,q)\simeq\End_{\mathcal O}(P_{KZ}\otimes_R\C)^{\rm op}$.
\item[(3)]
$\mathcal F=KZ(R)$ has a right adjoint functor 
$G:\H_n(\underline{\mathbf v},\mathbf q)\text{-}\fgMod
\rightarrow\mathcal O(R)$ and $\mathcal F\mathcal G\simeq\Id$. 
\end{itemize}
\end{theorem}

First, we consider the case that $R=\C[\underline{\bk},\bh]_{\mathfrak p}$ 
where $\mathfrak p=(\underline{\bk}-\underline\kappa,\bh-h)$. 
Then the KZ functor is representable by a projective object $P_{KZ}$ of 
$\mathcal O(R)$ by Lemma \ref{representable}. 
As $P_{KZ}$ is a right $\End_{\mathcal O(R)}(P_{KZ})^{\rm op}$-module, 
$KZ(P_{KZ})=\End_{\mathcal O(R)}(P_{KZ})$ is a 
$(\End_{\mathcal O(R)}(P_{KZ})^{\rm op},
\End_{\mathcal O(R)}(P_{KZ})^{\rm op})$-bimodule. Namely, 
$a\cdot\varphi\cdot b=ab\varphi$, for $\varphi\in KZ(P_{KZ})$ and 
$a,b\in\End_{\mathcal O(R)}(P_{KZ})^{\rm op}$. Using the right action, 
we may define a $R$-algebra homomorphism 
$$
\rho:\H_n(\underline{\mathbf v},\mathbf q)\rightarrow \End_{\mathcal O(R)}(P_{KZ})^{\rm op}
$$
by $h\cdot\varphi=\varphi\cdot\rho(h)(=\rho(h)\varphi)$, for 
$h\in \H_n(\underline{\mathbf v},\mathbf q)$ and $\varphi\in KZ(P_{KZ})$.
In particular, the $\H_n(\underline{\mathbf v},\mathbf q)$-module structure
on $KZ(P_{KZ})$ is the pullback of the $\End_{\mathcal O(R)}(P_{KZ})^{\rm op}$-module
structure via $\rho$.

Suppose that $(\underline\kappa,h)$ is generic enough, then $\mathcal O$ is a semisimple 
category and
$$
\End_{\H_n(\underline v,q)}(KZ(\Delta(E)))\simeq 
\End_{\mathcal O}(\Delta(E))=\C
$$
implies that $\{KZ(\Delta(E))\mid E\in\Irr W\}$ is 
a set of pairwise non-isomorphic simple $\H_n(\underline v,q)$-modules.
As $\dim_\C KZ(\Delta(E))=\dim E$, we have
$$
P_{KZ}\otimes_R\C=\oplus_{E\in\Irr W}\Delta(E)^{\oplus\dim E}
$$
and $\End_{\mathcal O}(P_{KZ}\otimes_R\C)\simeq 
\oplus_{E\in\Irr W}\End_\C(KZ(\Delta(E)))$. Therefore,
$$
\rho\otimes_R\C:\H_n(\underline v,q)\rightarrow 
\End_{\mathcal O}(P_{KZ}\otimes_R\C)^{\rm op}
$$
is surjective. Comparing the dimensions, we have
$$
\H_n(\underline{\mathbf v},\mathbf q)
\otimes_R\C\simeq \End_{\mathcal O}(P_{KZ}\otimes_R\C)^{\rm op}\simeq
\End_{\mathcal O(R)}(P_{KZ})^{\rm op}
\otimes_{\mathcal O(R)}\C
$$
and 
$\H_n(\underline{\mathbf v},\mathbf q)\otimes_R\C(\underline{\bk},\bh)\simeq
\End_{\mathcal O}(P_{KZ}\otimes_R\C(\underline{\bk},\bh))^{\rm op}$. 
Hence, we also have
$$
\H_n(\underline{\mathbf v},\mathbf q)\otimes_R K\simeq
\End_{\mathcal O}(P_{KZ}\otimes_R K)^{\rm op}. 
$$
Now we return to 
$R=\C[[\bk_1-\kappa_1,\dots,\bk_{d-1}-\kappa_{d-1},\bh-h]]$, for
arbitrary $(\underline\kappa,h)$. Then, we have 
a projective object $P_{KZ}$ of $\mathcal O(R)$ such that $P_{KZ}\otimes_R K$ also 
represents $KZ(K)$. Thus, $\rho:\H_n(\underline{\mathbf v},\mathbf q)\rightarrow 
\End_{\mathcal O(R)}(P_{KZ})^{\rm op}$ is injective. On the other hand, we have
$$
P_{KZ}\otimes_R\C=\oplus_{E\in\Irr W}P(E)^{\oplus\dim KZ(L(E))}.
$$
As $\End_{\mathcal O}(P_{KZ}\otimes_R\C)\simeq
\End_{\H_n(\underline v,q)}(KZ(P_{KZ}\otimes_R\C))$
by Lemma \ref{fully faithful},
$KZ(P(E))$, for $E$ such that $KZ(L(E))\neq 0$, are indecomposable
$\H_n(\underline v,q)$-modules and
$$
\Top KZ(P_{KZ}\otimes_R\C)=\sum_{E\in \Irr W}KZ(L(E))^{\oplus\dim KZ(L(E))}.
$$
Thus, $\Top\End_{\mathcal O}(P_{KZ}\otimes_R\C)$ is isomorhic to the
direct sum of $\End_\C(KZ(L(E)))$ over
$E$ such that $KZ(L(E))\neq 0$, and it follows that the composition map
$$
\H_n(\underline v,q)\rightarrow \End_{\mathcal O}(P_{KZ}\otimes_R\C)^{\rm op}
\rightarrow \Top\End_{\mathcal O}(P_{KZ}\otimes_R\C)^{\rm op}
$$
is surjective. Hence $\rho\otimes_R\C$ is surjective as well. 
This implies that $\rho$ is surjective, and (2) follows. 

\subsection{Faithful covers}

\begin{definition}
Let $R$ be a Noetherian commutative ring, $A$, $B$ module-finite $R$-algebras, 
$P$ a finitely generated projective $A$-module. Suppose that 
\begin{itemize}
\item[(i)]
$B\simeq\End_A(P)^{\rm op}$, 
\item[(ii)]
$A\text{-}\fgMod$ is a split highest weight category and projective $A$-modules 
are $\Delta$-filtered.
\end{itemize}
Then, $A$ (or $A\text{-}\fgMod$) is a \emph{$n$-faithful quasihereditary cover} of 
$B$ (or $B\text{-}\fgMod$) if  
$$
\mathcal F=\Hom_A(P,-):A\text{-}\fgMod\rightarrow B\text{-}\fgMod
$$
satisfies the condition
$$
\Ext_A^i(M,N)\simeq\Ext_B^i(\mathcal F(M),\mathcal F(N))\;\;
(0\leq i\leq n),
$$
for any $\Delta$-filtered $A$-modules $M, N$.

Let $S(\lambda)=\mathcal F(\Delta(\lambda))$, for $\lambda\in\Lambda$. 
We call them \emph{Specht modules}. 
\end{definition}

\begin{lemma}
\label{criterion}
Let $R$ be a Noetherian commutative ring, $A$, $B$ module-finite $R$-algebras, 
$P$ a finitely generated projective $A$-module, 
$\mathcal U$ an additive full subcategory of $A\text{-}\fgMod$. 
Suppose that 
$B\simeq\End_A(P)^{\rm op}$ and that $\mathcal U$ contains $A$. 
Define functors $\mathcal F=\Hom_A(P,-)$ and 
$\mathcal G=\Hom_B(\mathcal F(A),-)$. Then, we have the following. 
\begin{itemize}
\item[(1)]
The following are equivalent.
\begin{itemize}
\item[(a)]
$\Hom_A(M,N)\simeq\Hom_B(\mathcal F(M),\mathcal F(N))$, for 
$M,N\in\mathcal U$. 
\item[(b)]
$M\simeq\mathcal G(\mathcal F(M))$, for $M\in\mathcal U$. 
\end{itemize}
\item[(2)]
Suppose that $\Hom_A(M,N)\simeq\Hom_B(\mathcal F(M),\mathcal F(N))$, for 
$M,N\in\mathcal U$. Then, the following are equivalent.
\begin{itemize}
\item[(a)]
$\Ext^i_A(M,N)\simeq\Ext^i_B(\mathcal F(M),\mathcal F(N))$ $(i=0,1)$, for 
$M,N\in\mathcal U$. 
\item[(b)]
$R^1\mathcal G(\mathcal F(M))(=\Ext^1_B(\mathcal F(A),\mathcal F(M)))=0$, 
for $M\in\mathcal U$. 
\end{itemize}
\end{itemize}
\end{lemma}

Let $R$ be a regular local ring whose residue field contains $\C$, 
$A$ a module-finite projective $R$-algebra, 
$P$ a finitely generated projective $A$-module. We denote 
$$
\mathcal F=\Hom_{\mathcal C}(P,-):\mathcal C\rightarrow B\text{-}\fgMod,
$$
$B=\End_{\mathcal C}(P)^{\rm op}$ and 
$\mathcal C=A\text{-}\fgMod$, 
$\mathcal C(\mathfrak p)=A\otimes_R(R/\mathfrak p)\text{-}\fgMod$, for 
$\mathfrak p\in\Spec R$. We also denote the quotient field of $R/\mathfrak p$ by 
$Q(\mathfrak p)$. 

We consider the following conditions, for $\mathfrak p\in\Spec R$ such that 
$R/\mathfrak p$ is regular and $A\otimes_R Q(\mathfrak p)$ 
is split semisimple. 
\begin{itemize}
\item[(I)]
There exist a finite preordered set $\Lambda$ and 
$\{\Delta(\lambda)\in\mathcal C\mid\lambda\in\Lambda\}$ such that 
\begin{itemize}
\item[(i)]
$\Delta(\lambda)$ is $R$-free, for $\lambda\in\Lambda$. 
\item[(ii)]
$\mathcal C(\mathfrak p)$ is a split highest weight category 
whose standard objects are 
$$
\{\Delta(\lambda)\otimes_R(R/\mathfrak p)\in\mathcal C\mid\lambda\in\Lambda\}.
$$
\item[(iii)]
Projective objects of $\mathcal C(\mathfrak p)$ are $\Delta$-filtered. 
\end{itemize}
\item[(II)]
$B$ satisfies the following.
\begin{itemize}
\item[(i)]
$B\otimes_R(R/\mathfrak p)\simeq\End_{\mathcal C}(P\otimes_R(R/\mathfrak p))^{\rm op}$.
\item[(ii)]
$\mathcal F$ restricts to
$\mathcal F(\mathfrak p)=\Hom_{\mathcal C}(P\otimes_R(R/\mathfrak p),-):
\mathcal C(\mathfrak p)\rightarrow B\otimes_R(R/\mathfrak p)\text{-}\fgMod$. 
\end{itemize}
\item[(III)]
If $M\in\mathcal C(\mathfrak p)$ is $\Delta$-filtered then 
$\mathcal F(\mathfrak p)(M)$ is a finitely generated projective $R/\mathfrak p$-module. 
\end{itemize}
We also consider the following conditions. 
\begin{itemize}
\item[(IV)]
$\mathcal F(\mathfrak m):A\otimes_R(R/\mathfrak m)\text{-}\fgMod\rightarrow
B\otimes_R(R/\mathfrak m)\text{-}\fgMod$ is a $0$-faithful cover, for the maximal 
ideal $\mathfrak m\in\Spec R$. 
\item[(V)]
Let $K$ be the quotient field of $R$. Then $A\otimes_RK$ is a split semisimple 
$K$-algebra.
\end{itemize}

If $R=\C[[\bk_1-\kappa_1,\dots,\bk_{d-1}-\kappa_{d-1},\bh -h]]$ and 
suppose $v_i=\exp(2\pi\sqrt{-1}(\kappa_i+\frac{i}{d}))$, 
$1\leq i\leq d-1$, are different from 
$1$ and pairwise distinct, $q=\exp(2\pi\sqrt{-1}h)\neq-1$, then we already 
know that $\mathcal O(R)\simeq A\text{-}\fgMod$, for some module-finite 
projective $R$-algebra $A$, such that $A\otimes_R\C$ is a $0$-faithful cover of 
$\H_n(\underline v,q)$ and the above conditions are satisfied. 
The following is a key result \cite[Proposition 4.42]{R}. 

\begin{proposition}
\label{0to1-faithful}
Let $R$ be a regular local ring whose residue field contains $\C$, 
$A$ a module-finite projective $R$-algebra, 
$P$ a finitely generated projective $A$-module such that 
the conditions (I)-(V) are satisfied. Then $A$ is a $1$-faithful cover of $B$. 
\end{proposition}
\begin{proof}
The proof is induction on $\dim R$. If $\dim R=0$ then $A$ is split semisimple 
and $\Ext_B^1(\mathcal F(A),\mathcal F(M))=0$, for $M\in A\text{-}\fgMod$. 
Suppose $\dim R>0$. 
Let $M\in A\text{-}\fgMod$ be $\Delta$-filtered. 
We argue that $\Ext^1_B(\mathcal F(A),\mathcal F(M))\neq 0$ leads to a 
contradiction. Define 
$$
Z=\supp \Ext^1_B(\mathcal F(A),\mathcal F(M)), 
$$
which we write $V(\mathfrak p_1)\cup\cdots\cup V(\mathfrak p_r)$, where 
$\mathfrak p_1,\dots,\mathfrak p_r$ are minimal elements of
$\Ass(\Ext^1_B(\mathcal F(A),\mathcal F(M)))$, the
set of associated primes of $\Ext^1_B(\mathcal F(A),\mathcal F(M))$. 
As $P$ and $\mathcal F(A)$ are finitely presented, 
$$
\Ext^1_B(\mathcal F(A),\mathcal F(M))_{\mathfrak p}=
\Ext^1_B(\mathcal F(A_{\mathfrak p}),\mathcal F(M_{\mathfrak p})). 
$$
Thus, we may assume $r=1$ and $Z=V(\mathfrak m)$, where $\mathfrak m$ is 
the maximal ideal of $R$, without loss of generality. Take an open set 
$$
D(\alpha)\subseteq
\{\mathfrak p\in\Spec R\mid \text{$\mathcal C(\mathfrak p)$ is semisimple.}\}\neq
\emptyset.
$$
Then, $Z\subseteq V(\alpha)$. Let 
$\mathfrak m=(x_1,\dots,x_d)$ be a regular system of parameters. As $R$ is 
regular, it is UFD, and the intersection of the ideals $(\pi)$, for all 
$\pi\in\mathfrak m\setminus\mathfrak m^2$, is $0$. Hence we may choose 
$\pi\in\mathfrak m\setminus\mathfrak m^2$ such that $\alpha\not\in(\pi)$. 
Note that $R/(\pi)$ is a regular local ring. By $\alpha\not\in(\pi)$, 
we have $D(\alpha)\cap V(\pi)\neq\emptyset$. Thus, (V) holds for $R/(\pi)$. 
(I)-(IV) clearly hold for $R/(\pi)$. Therefore, $A\otimes_R(R/(\pi))$ is a 
$1$-faithful cover of $B\otimes_R(R/(\pi))$ by the induction hypothesis. Let 
$Q_\bullet\rightarrow \mathcal F(A)\rightarrow 0$ be a free resolution, and 
set $C^\bullet=\Hom_B(Q_\bullet,\mathcal F(M))$. 
By Proposition \ref{key prop}, $C^i$ are flat $R$-modules. Thus, 
for $\mathcal G=\Hom_B(\mathcal F(A),-)$, we have 
the K\"unneth spectral sequence
$$
E_2^{p,q}=\bigoplus_{-i+j=q}
\Tor_p^R(R^{-q}\mathcal G(\mathcal F(M)),R/(\pi))\Longrightarrow 
R^{p+q}\mathcal G(\mathcal F(M\otimes_R(R/(\pi))),
$$
since $H^{p+q}(C^\bullet\otimes_R(R/(\pi)))\simeq
R^{p+q}\mathcal G(\mathcal F(M\otimes_R(R/(\pi)))$. The spectral sequence 
degenerates at the $E_2$ page, and we obtain
\begin{multline*}
0\rightarrow R^n\mathcal G(\mathcal F(M))\otimes_R(R/(\pi))\rightarrow
R^n\mathcal G(\mathcal F(M\otimes_R(R/(\pi))))\\
\rightarrow
\Tor_1^R(R^{n+1}\mathcal G(\mathcal F(M)),R/(\pi))\rightarrow0.
\end{multline*}
We only need the exact sequence for $n=0$ below. Define
$$
\varphi:\mathcal G\mathcal F(M)\otimes_R(R/(\pi))\rightarrow
\mathcal G(\mathcal F(M)\otimes_R(R/(\pi)))\simeq
\mathcal G\mathcal F(M\otimes_R(R/(\pi))). 
$$
Then, the composition map 
$$
M\otimes_R(R/(\pi))\rightarrow \mathcal G\mathcal F(M)\otimes_R(R/(\pi))
\overset{\varphi}{\rightarrow}\mathcal G\mathcal F(M\otimes_R(R/(\pi)))
$$
is an isomorphism by the induction hypothesis and Lemma \ref{criterion}(1). 
Hence, $\varphi$ is surjective. Using the exact sequence above, we have 
\begin{itemize}
\item[(a)]
$M\otimes_R(R/(\pi))\simeq \mathcal G\mathcal F(M)\otimes_R(R/(\pi))
\overset{\varphi}{\simeq} 
\mathcal G\mathcal F(M\otimes_R(R/(\pi)))$,
\item[(b)]
$\Tor_1^R(R^1\mathcal G(\mathcal F(M),R/(\pi)))=0$.
\end{itemize}
(b) implies that 
$$
0\rightarrow R^1\mathcal G(\mathcal F(M))\overset{\pi\cdot}{\rightarrow}
R^1\mathcal G(\mathcal F(M))\rightarrow
R^1\mathcal G(\mathcal F(M))\otimes_R(R/(\pi))\rightarrow0.
$$
However, $\mathfrak m\in\Ass(R^1\mathcal G(\mathcal F(M)))$ implies that 
there exists $x\in R^1\mathcal G(\mathcal F(M))$ such that 
$\operatorname{Ann}_R(x)=\mathfrak m$, and $\pi\in\mathfrak m$ implies that 
$\pi x=0$. This is absurd. 
We have proved that $R^1\mathcal G(\mathcal F(M))=0$, for any $\Delta$-filtered 
$M\in A\text{-}\fgMod$. 

Since $\mathcal F(M)$ is projective as an $R$-module, we have 
$$
0\rightarrow\mathcal F(M)\overset{\pi\cdot}{\rightarrow}
\mathcal F(M)\rightarrow \mathcal F(M)\otimes_R(R/(\pi))\rightarrow0.
$$
We apply the functor $\mathcal G$ to the exact sequence and use (a). Then, 
$$
0\rightarrow\mathcal G\mathcal F(M)\overset{\pi\cdot}{\rightarrow}
\mathcal G\mathcal F(M)\rightarrow
\mathcal G\mathcal F(M)\otimes_R(R/(\pi))\rightarrow0
$$
and we have a morphism of exact sequences from 
$$
0\rightarrow M \overset{\pi\cdot}{\rightarrow}
M\rightarrow M\otimes_R(R/(\pi))\rightarrow0. 
$$
Let $K=\Ker(M\rightarrow \mathcal G\mathcal F(M))$. Then, we may show that 
$$
0\rightarrow K\otimes_R(R/(\pi))\rightarrow M\otimes_R(R/(\pi))
\simeq \mathcal G\mathcal F(M)\otimes_R(R/(\pi)).
$$
Thus, $K\otimes_R(R/(\pi))=0$ and Nakayama's lemma implies that $K=0$. Now, 
we consider $M\subseteq \mathcal G\mathcal F(M)$ and use (a) again. 
Then $(\mathcal G\mathcal F(M)/M)\otimes_R(R/(\pi))=0$ and Nakayama's 
lemma implies $\mathcal G\mathcal F(M)\simeq M$. Lemma \ref{criterion}(1) 
implies that $A$ is a $0$-faithful cover of $B$, and 
$A$ is a $1$-faithful cover of $B$ by Lemma \ref{criterion}(2). 
\end{proof}

\subsection{Uniqueness of quasihereditary covers}

Now we are prepared to state main results in \cite{R}.
Let us start with definitions and consequences of Proposition \ref{0to1-faithful}.

\begin{definition}
Suppose that $\mathcal C$ and $\mathcal C'$ are highest weight categories over $R$. 
Let $\{\Delta(\lambda)\mid \lambda\in\Lambda\}$ and 
$\{\Delta(\lambda')\mid \lambda'\in\Lambda'\}$ be standard objects of 
$\mathcal C$ and $\mathcal C'$ respectively. 
We say that $\mathcal C$ and $\mathcal C'$ are 
\emph{equivalent highest weight categories} if there is an equivalence 
$\mathcal C\simeq \mathcal C'$ and a bijection $\Lambda\simeq\Lambda'$ 
such that if $\lambda$ corresponds to $\lambda'$ then 
$\Delta(\lambda)$ goes to $\Delta(\lambda')\otimes_R U$, for 
an invertible $R$-module $U$.
\end{definition}

\begin{definition}
Let $R$ be a Noetherian commutative ring, $B$ a module-finite $R$-algebra. 
Let $\mathcal F:A\text{-}\fgMod\rightarrow B\text{-}\fgMod$ and 
$\mathcal F':A'\text{-}\fgMod\rightarrow B\text{-}\fgMod$ be quasihereditary covers. 
If there is an equivalence of highest weight categories 
$\mathcal K:A\text{-}\fgMod\simeq A'\text{-}\fgMod$ such that 
$\mathcal F'\mathcal K=\mathcal F$, we say that 
$A$ and $A'$ are \emph{equivalent quasihereditary covers}.
\end{definition}

\begin{theorem}[{\cite[Proposition 4.44, Corollary 4.45]{R}}]
Let $R$ be a Noetherian commutative ring, $B$ a module-finite $R$-algebra. 
Let $\mathcal F:A\text{-}\fgMod\rightarrow B\text{-}\fgMod$ and 
$\mathcal F':A'\text{-}\fgMod\rightarrow B\text{-}\fgMod$ be 
$1$-faithful quasihereditary covers. 
If the set of Specht modules coincide, that is, 
$\{S(\lambda)\mid \lambda\in\Lambda\}=\{S(\lambda')\mid \lambda\in\Lambda'\}$, 
then $A$ and $A'$ are equivalent quasihereditary covers.
\end{theorem}

\begin{theorem}[{\cite[Theorem 4.48]{R}}]
\label{uniqueness}
Let $R$ be a Noetherian commutative domain, $K$ its quotient field. 
Suppose that $B$ is a module-finite projective $R$-algebra such that 
$B\otimes_RK$ is split semisimple. 
Let $\mathcal F:A\text{-}\fgMod\rightarrow B\text{-}\fgMod$ and 
$\mathcal F':A'\text{-}\fgMod\rightarrow B\text{-}\fgMod$ be 
$1$-faithful quasihereditary covers. 
Suppose that the preorders on $\Irr B$ are compatible. Then
$A$ and $A'$ are equivalent quasihereditary covers.
\end{theorem}

\subsection{The category $\mathcal O$ as quasihereditary covers}

Rouquier's motivation for developing the theory of quasihereditary covers 
which we have explained so far, is to prove the following. My motivation 
to write this survey is to explain and advertise his beautiful ideas to 
prove the theorem.

\begin{theorem}[Rouquier]
\label{main theorem}
Suppose that $v_i=\exp(2\pi\sqrt{-1}(\kappa_i+\frac{i}{d}))$, 
$1\leq i\leq d-1$, are different from 
$1$ and pairwise distinct, $q=\exp(2\pi\sqrt{-1}h)\neq-1$. 
Then the category $\mathcal O$ for the rational Cherednik algebra 
$H_\C(\underline\kappa,h)$ is a quasihereditary cover of the Hecke 
algebra $\H_n(\underline v,q)$. If it is the Hecke algebra associated 
with the symmetric group, then $\mathcal O$ and the module category of 
the $q$-Schur algebra over $\C$ are equivalent quasihereditary covers. 
\end{theorem}

The idea of the proof of the second part is to lift the $0$-faithful cover
over $\C$ to a $1$-faithful cover 
over $R=\C[[\bk_1-\kappa_1,\dots,\bk_{d-1}-\kappa_{d-1},\bh -h]]$, by 
using Proposition \ref{0to1-faithful}, and 
apply the uniqueness result Theorem \ref{uniqueness}.

As we explained in 2.4,
Theorem \ref{main theorem} is closely related to the Fock space theory 
in the second part, 
categorification of JMMO deformed Fock spaces, and the results of Geck and Jacon 
on the canonical basic sets in the first part, cellular structures of 
Hecke algebras. See \cite[6.5]{R} for some conjectures in this field.
I also recommend reading papers \cite{G} and \cite{GM} by Iain Gordon.

\end{document}